\documentclass[12pt]{amsart}

\RequirePackage[colorlinks,citecolor=blue,urlcolor=blue]{hyperref}

\usepackage{amsmath, amsthm, amsfonts, amsbsy, amssymb, upref, hyperref, pdfsync, enumerate} 
\usepackage{epsfig, graphicx}
\usepackage{color}
\usepackage{url}
\usepackage{xy}
\usepackage{ftnxtra}
\usepackage{fnpos}
\usepackage{tikz}
\usepackage[colorinlistoftodos,color=red!70]{todonotes}
\usepackage[margin=1in]{geometry}

\theoremstyle{theorem}
\newtheorem{thm}{Theorem}
\newtheorem{lem}[thm]{Lemma}
\newtheorem{cor}[thm]{Corollary}

\newtheorem{prop}[thm]{Proposition}

\theoremstyle{remark}
\newtheorem{defn}[thm]{Definition}
\newtheorem{rem}[thm]{Remark}

\newtheorem{ex}[thm]{Example}
\numberwithin{thm}{section} \numberwithin{equation}{section}

\newcommand{\tp}{^{\mathsf{T}}}
\newcommand{\RR}{\mathbb{R}}
\renewcommand{\SS}{\mathbb{S}}
\newcommand{\eps}{\varepsilon}
\newcommand{\Prob}{\mathbb{P}}
\newcommand{\EE}{\mathbb{E}}

\newcommand{\PP}{\mathbb{P}} 
\newcommand{\mcN}{\mathcal{N}}

\newcommand{\X}{\mathcal{X}}
\newcommand{\mcX}{\mathcal{X}}

\newcommand{\Trace}{\operatorname{Tr}}

\DeclareMathOperator{\Sec}{Sec}

\DeclareMathOperator{\grad}{grad}
\DeclareMathOperator{\Hess}{Hess}
\DeclareMathOperator{\Lip}{Lip}

\DeclareMathOperator{\pos}{pos}
\DeclareMathOperator{\vel}{mom}
\DeclareMathOperator{\Var}{Var}

\DeclareMathOperator{\InfSec}{Inf~Sec}
\DeclareMathOperator{\diam}{diam}

\newcommand\Rone[1]{\textcolor{gray}{R1: #1}}
\newcommand\Rtwo[1]{\textcolor{gray}{R2: #1}}
\newcommand\christof[1]{\textcolor{red}{CS: #1}}
\newcommand\simon[1]{\textcolor{cyan}{SR: #1}}

\newcommand{\vanish}[1]{}

\begin{document}


\title[Curvature and Concentration of HMC]{Curvature and Concentration of Hamiltonian Monte Carlo in High Dimensions}

\author{Susan Holmes}\thanks{Susan Holmes is supported by NIH grant R01-GM086884.}
\address{Department of Statistics, Stanford University, 390 Serra Mall, Stanford CA 94305}
\email{susan@stat.stanford.edu}
\author{Simon Rubinstein-Salzedo}\thanks{Simon Rubinstein-Salzedo is supported by NIH grant R01-GM086884.}
\email{simonr@stanford.edu}
\author{Christof Seiler}\thanks{Christof Seiler is supported by a postdoctoral fellowship from the Swiss National Science Foundation and a travel grant from the France-Stanford Center for Interdisciplinary Studies.}
\email{christof.seiler@stanford.edu}



\date{\today}
\maketitle

\begin{abstract}
In this article, we analyze Hamiltonian Monte Carlo (HMC) by placing it in the setting of Riemannian geometry using the Jacobi metric, so that each step corresponds to a geodesic on a suitable Riemannian manifold. We then combine the notion of curvature of a Markov chain due to Joulin and Ollivier with the classical sectional curvature from Riemannian geometry to derive error bounds for HMC in important cases, where we have positive curvature. These cases include several classical distributions such as multivariate Gaussians, and also distributions arising in the study of Bayesian image registration.
The theoretical development suggests the sectional curvature as a new diagnostic tool for
convergence for certain Markov chains.
\end{abstract}





\section{Introduction} \label{sec:intro}

Approximating integrals is central to most statistical endeavors. Here, we investigate an approach drawing from probability theory, Riemannian geometry, and physics that has applications in MCMC generation of posterior distributions for biomedical image analyses.

We take a measure space $\mcX$, a function $f:\mcX\to\RR$ (which for our purposes will be assumed Lipschitz), and a probability distribution $\pi$ on $\mcX$, and we aim to approximate \[I=\int_{\mcX} f(x)\, \pi(dx).\] One way to do so is to pick a  number $T$, choose points $x_1,x_2,\ldots,x_T\in\mcX$ sampled according to $\pi$, and estimate $I$ by \begin{equation} \label{approxintegral} \widehat{I} = \frac{1}{T}\sum_{i=1}^T f(x_i).\end{equation}

However, difficulties quickly arise. How do we sample from $\pi$? How do we select $T$ so that the error is below an acceptable threshold without having to choose  $T$ so big that computation is prohibitively time-consuming? And how do we bound the error? In this article, we address these issues.

\subsection{Main Contribution}
The goal of this article is to compute error bounds for $\PP(|I-\widehat{I}|\ge r)$ when points $x_1,x_2,\ldots,x_T\in\mcX$ are approximately sampled from $\pi$ using Hamiltonian Monte Carlo (see \S\ref{sec:hamiltonian}). Our bounds are applications of  theorems of Joulin and Ollivier in \cite{JO10} using a special notion of curvature. Our main  contribution is the combined use of the Jacobi metric  \cite{Pin75} (see \S\ref{sec:geometry}) with the curvature approach introduced by Joulin and Ollivier in~\cite{Joulin2007,JO10,Ollivier09}.

Our results specifically target the \emph{high-dimensional} setting in which there is no underlying low-dimensional model. 
We can show that, in important classes of distributions, $T$ 
depends only polynomially on the dimension and error tolerance. This is
relevant for the modern world where there are multiple sources
of high-dimensional data. For instance, medical imaging, 
for which we provide an example in \S\ref{sec:anatomy},
produces high-dimensional data, as points in one region of an image are essentially independent of those in other regions.

\subsection{Background}
The idea of constructing a Markov chain whose stationary
distribution is $\pi$ appeared originally in a five-author paper
\cite{Metropolis53}
which introduced the incremental random walk proposal, generalized
by Hastings to include 
independent proposals \cite{Hastings70}.
A recent overview of the subject can be found in \cite{Diaconis09}.

 
Unfortunately, questions about error bounds are more difficult, and in practice, the Metropolis-Hastings algorithm  can converge slowly. One reason is that the algorithm uses minimal information, considering only the probabilities of the proposal distribution and a comparison of probabilities of the target distribution. In particular, it does not use any geometric information.
It stands to reason that an algorithm that sees more of the structure of the problem ought to perform better.

We are going to explore a different generation method inspired from physics called Hamiltonian Monte Carlo. We imagine an object moving around on $\mcX$ continuously. From time to time, we measure the position $x_i$ of the object, with the aim of using these discrete measurements $x_i$ in (\ref{approxintegral}). We can imagine
using $\pi$ to determine the time between measurements, thereby ``distorting time'' based on $\pi$. 
In regions of high density, we increase the measurement frequency, 
so that we obtain many samples from these regions. In regions of low density, 
we decrease the measurement frequency,
so that we obtain few samples there.

An equivalent approach, which will make the link with 
Riemannian and differential geometry, is to think of $\pi$ as stretching and shrinking the space $\mcX$ so that regions of high density are physically larger, and regions of low density are physically smaller. These two approaches are in fact the same, as shrinking a region while keeping the time between measurements fixed has the same effect as keeping space uniform while varying the time between measurements.

The idea of stretching and shrinking space is nothing new in probability and statistics; for instance, inverse transform sampling for a distribution on $\RR$ samples a point $p$ from a distribution with cumulative distribution function $F$ by picking a uniform random number $x\in [0,1]$ and letting $p$ be the largest number so that $F(p)\le x$, as in Figure~\ref{fig:CDF_Cauchy}. Here, we are shrinking the regions of low density so that they are less likely to be selected.

\begin{figure} 
    \centering
    \includegraphics[width=0.4\textwidth]{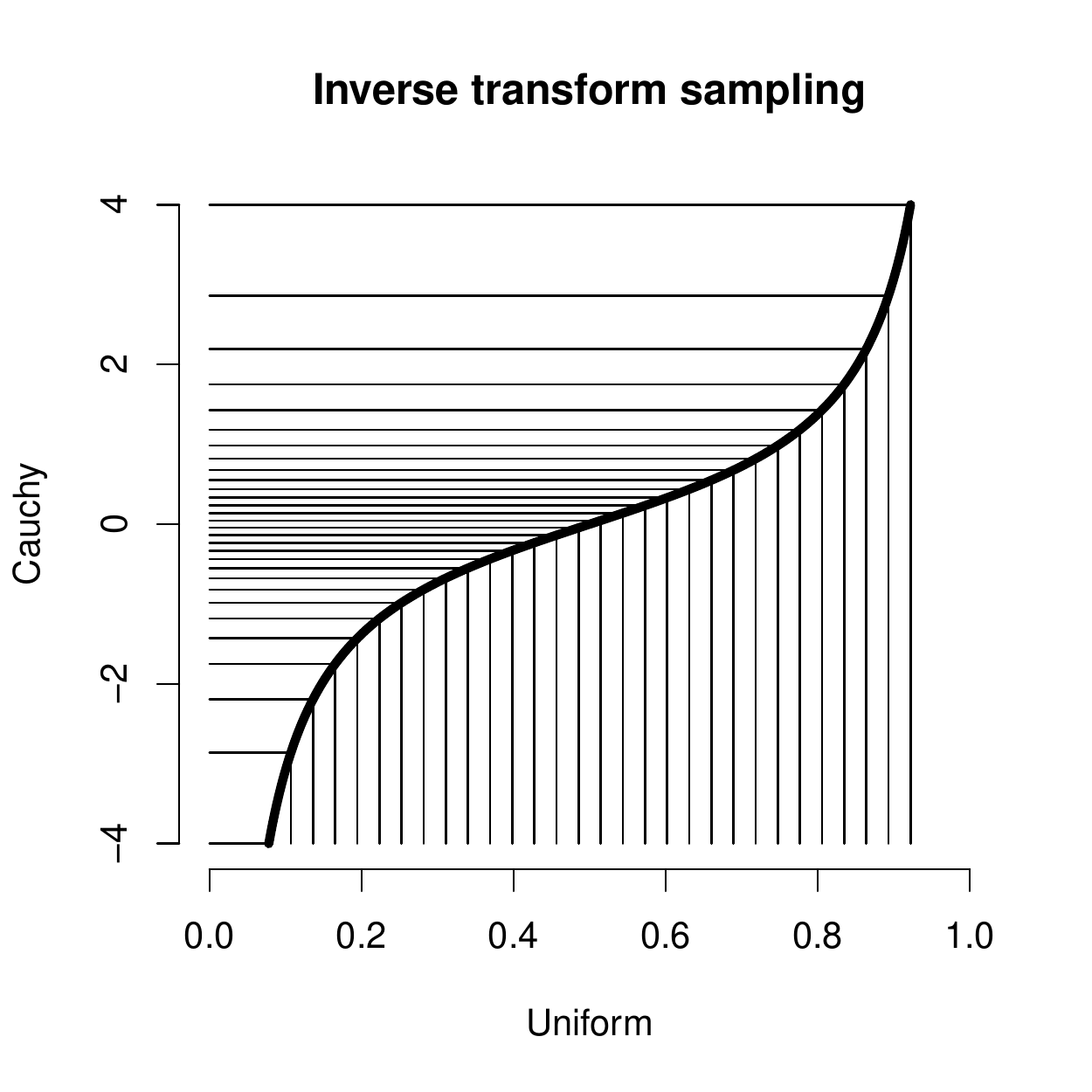}
    \caption{Inverse transform sampling of a standard Cauchy distribution.}
    \label{fig:CDF_Cauchy}
\end{figure}

\begin{ex} Consider the Cauchy distribution, which has cumulative distribution function $F(p)=\frac{1}{2}+\frac{1}{\pi}\arctan(p)$. Its inverse function is $F^{-1}(x)=\tan\left(\pi x-\frac{\pi}{2}\right)=-\cot(\pi x)$. To sample from this distribution, we pick $x\in[0,1]$ uniformly, and then we let $p=F^{-1}(x)$. Then $p$ is a Cauchy-random variable. This method is illustrated in Figure~\ref{fig:CDF_Cauchy}. \end{ex}

In order to start the process, we  put the particle in an initial position and start moving it. 
In our applications, we will assume that the starting point $x_0$ will be chosen using prior information, for instance using the mode of the
prior distribution in Bayesian computations. This is called a \emph{warm start} in the literature~\cite{Lovasz99}.



We review precise general statements from Joullin and Olliver's work~\cite{JO10} in \S\ref{sec:concentration},  in particular, Theorem~\ref{thm:ConcentrationInequality}, and then apply them to our case in \S\ref{sec:examples}

In \S\ref{sec:geometry}, we define a notion of curvature for Markov chains, in \S\ref{sec:concentration}, we use it to deduce error bounds for general Markov chains following \cite{JO10}, and in \S\ref{sec:examples}, we show new error bounds related to Markov chains motivated by the aforementioned physics analogy for three examples: a theoretical analysis of the multivariate Gaussian distribution in \S\ref{sec:GaussianExample} and the $t$ distribution \S\ref{sec:tDistribution}, and a real world example from medical image registration in \S\ref{sec:anatomy}. 

This article fulfills two goals. On the one hand, we produce new results on error bounds, which we believe to be of theoretical interest. On the other hand, we hope that it can serve as a user's guide for researchers in other areas of statistics
hoping to access  new tools from Riemannian geometry. We have made every effort to keep our presentation as  concrete as possible.

This article extends an early short version published as a conference article in the proceedings at NIPS \cite{SRH14a}. We present new introductory and background material, detailed mathematical developments of the proofs, and two new examples: the multivariate $t$ distribution and medical image registration.

\section*{Acknowledgements}

The authors would like to thank Sourav Chatterjee, Otis Chodosh, Persi Diaconis, Emanuel Milman, Veniamin Morgenshtern, Richard Montgomery, Yann Ollivier, Xavier Pennec, Mehrdad Shahshahani, and Aaron Smith for their insight and helpful discussions.

\section{Markov Chain Monte Carlo} \label{sec:MCMC}


Our goal in this article is to quantify the error made when
approximating \[I=\int_{\mcX} f\, d\pi \qquad \text{by} \qquad \widehat{I}=\frac{1}{T}\sum_{i=T_0+1}^T f(x_i),\] where $x_i$ are sampled using a special Markov chain
whose stationary distribution is
$\pi$ and $T_0$ represents a time at which a satisfactory soft starting point has been reached. 
The standard Metropolis-Hastings algorithm \cite{Metropolis53,Hastings70} 
uses  a proposal distribution $P_x$ starting  at $x\in\mcX$ and has
an acceptance probability $\alpha$ computed from the target and proposal.



\begin{rem} 
Using the Metropolis-Hastings algorithm replaces the task of sampling one point from $\pi$ directly with the task of sampling many times from the (potentially much simpler) proposal distributions $\{P_x\}_{x\in\mcX}$. 
\end{rem}

The Metropolis-Hastings (MH) algorithm provides flexibility in the choice of $P_x$. In practice, it is common to let $P_x$ be a Gaussian distribution centered at $x$, or a uniform distribution on a ball centered at $x$. It is necessary to compromise between high acceptance probabilities $\alpha$ and large variances of $P_x$. In order to force $\alpha\approx 1$, we can take  tiny steps, so that $P_x$ is highly concentrated near $x$. However, many steps are necessary to explore $\mcX$ thoroughly. On the other hand, $P$ can be chosen to move quickly at the cost of rarely accepting.

Gelman, Roberts, and Gilks \cite{GRG96} show that, in the case of a Gaussian $\mcN(0,I_d)$
target distribution   and spherically symmetric proposal distributions, 
the optimal proposal distribution, i.e.\ achieving the fastest mixing time, has standard deviation roughly $2.38/\sqrt{d}$ and acceptance probability roughly $0.234$ as $d\to\infty$. Since the step size goes to 0 as $d\to\infty$, it takes many steps to sample in large dimensions. 

In contrast, Hamiltonian Monte Carlo (HMC) (\S\ref{sec:hamiltonian}), a variant of the MH algorithm, allows us to overcome the issue of low acceptance probabilities. Beskos, Pillai, Roberts, Sanz-Serna, and Stuart \cite{Beskos13} show that, in the case of product distributions (which has as a special case the multivariate Gaussian with identity covariance), to reach $O(1)$ acceptance probability as the $d \to \infty$ one needs to scale the the step size by a factor of $O(d^{-1/4})$ as opposed to $O(d^{-1/2})$ for the MH algorithm. In addition, in many practical applications HMC has proven to be very efficient.  

For both MH and HMC, we need to define the number of steps $T$ that are required for $I$ to get close to $\widehat{I}$. 
We will analyze HMC in this article to give guidance on how large $T$ needs to be under suitable assumptions. 
Our focus here is in the computation of the variance component
(the second part of the right hand side) in the mean square error: 
\[
\EE_x(|\hat{I}-I|^2)= |\EE_x(\hat{I})-I|^2 + \Var_x \hat{I} .
\]
We first present an appropriate setting for HMC, which involves some Riemannian geometry.

\section{Riemannian Manifolds} \label{sec:riemann}

We introduce what we need for \S\ref{sec:hamiltonian} from differential and Riemannian geometry, saving ideas about curvature for manifolds and probability measures for \S\ref{sec:geometry}.
We go through the necessary material here rather quickly and we invite the interested reader to consult \cite{docarmo92} or a similar reference for a more thorough exposition.


\begin{defn} Let $\mcX$ be a $d$-dimensional manifold, and let $x\in \mcX$ be a point. Then the tangent space $T_x\mcX$ consists of all $\gamma'(0)$, where $\gamma:(-\eps,\eps)\to \mcX$ is a smooth curve and $\gamma(0)=x$. (See Figure \ref{fig:tangentspace}.) The tangent bundle $T\mcX$ of $\mcX$ is the manifold whose underlying set is the disjoint union $\bigsqcup_{x\in \mcX} T_x\mcX$. \end{defn}

\begin{figure}
    \centering
    \includegraphics[width=0.6\textwidth]{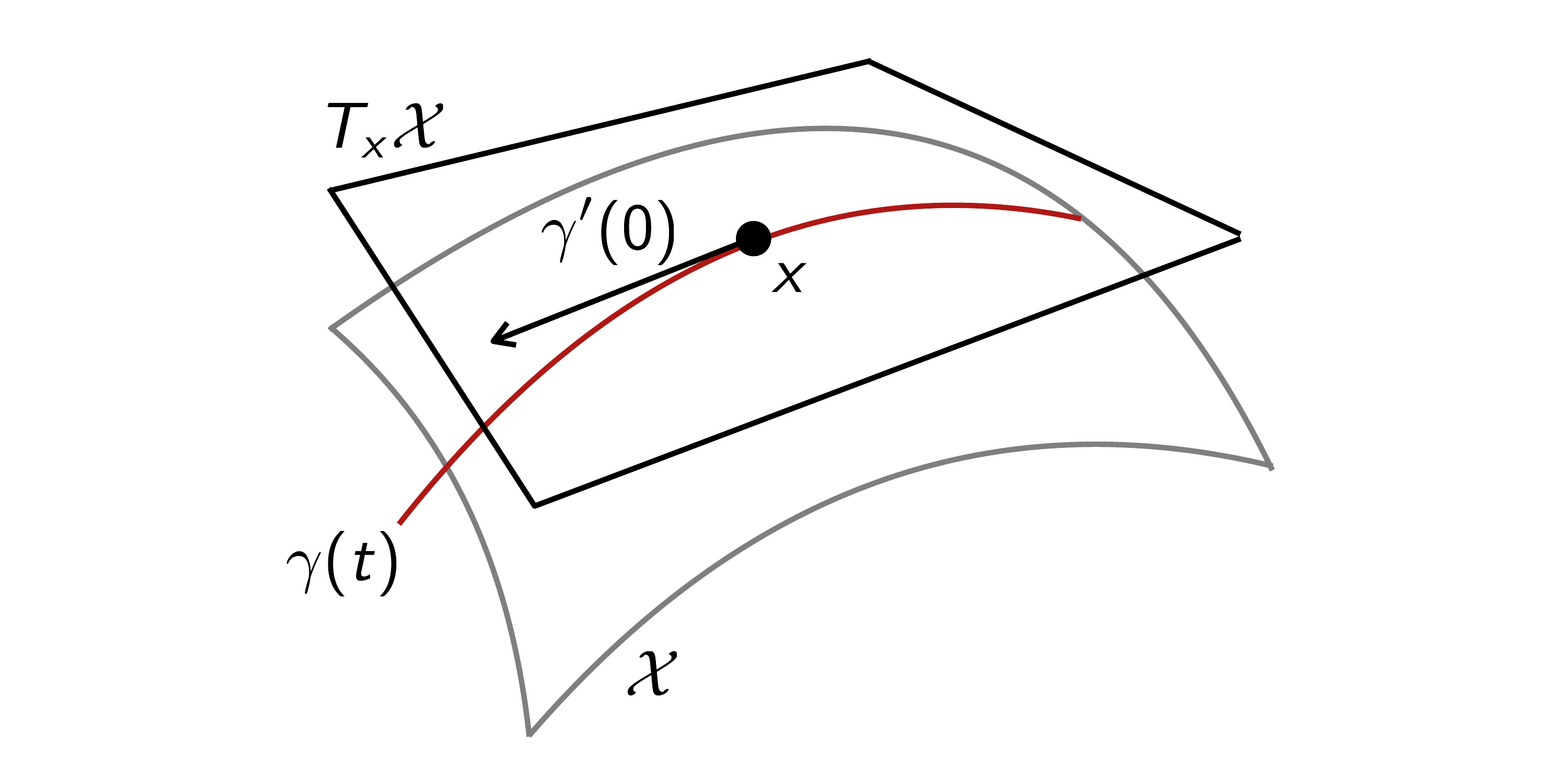}
    \caption{The tangent space $T_x\mcX$ to $\mcX$ at $x$.}
    \label{fig:tangentspace}
\end{figure}

\begin{rem} We can stitch $T_x\mcX$ and $T\mcX$ into manifolds. The details of that construction can be found in \cite{docarmo92}. For us, it suffices to note that $T_x\mcX$ is a vector space of dimension $d$, and $T\mcX$ is a manifold of dimension $2d$. \end{rem}

\begin{defn} 
A Riemannian manifold is a pair $(\mcX,\langle\cdot,\cdot\rangle)$, where $\mcX$ is a smooth ($C^\infty$) manifold and $\langle\cdot,\cdot\rangle$ is a positive definite bilinear form on each tangent space $T_x\mcX$, which varies smoothly with $x\in\mcX$. We call $\langle\cdot,\cdot\rangle$ the (Riemannian) metric. \end{defn}

The Riemannian metric allows us to measure distances between two points on $\mcX$. We define the \emph{length} of a curve $\gamma:[a,b]\to \mcX$ to be \[L(\gamma) = \int_a^b \langle \gamma'(t),\gamma'(t)\rangle^{1/2}\, dt,\] and the \emph{distance} $\rho(x,y)$ to be \[\rho(x,y)=\inf_{\substack{\gamma(0)=x \\ \gamma(1)=y}} L(\gamma).\]

A \emph{geodesic} on a Riemannian manifold is a curve $\gamma:[a,b]\to \mcX$ that locally minimizes distance, in the sense that if $\widetilde{\gamma}:[a,b]\to \mcX$ is another path with $\widetilde{\gamma}(a)=\gamma(a)$ and $\widetilde{\gamma}(b)=\gamma(b)$ with $\widetilde{\gamma}(t)$ and $\gamma(t)$ sufficiently close together for each $t\in [a,b]$, then $L(\gamma)\le L(\widetilde{\gamma})$.

\begin{ex} On $\RR^d$ with the standard metric, geodesics are exactly the line segments, since the shortest path between two points is along a straight line. On $\SS^d$, the geodesics are exactly segments of great circles. \end{ex}

In this article, we are primarily concerned with the case of $\mcX=\RR^d$. However, it will be essential to think in terms of Riemannian manifolds, as our metric on $\mcX$ will vary from the standard metric. In \S\ref{sec:geometry}, we will see how to choose a metric, the Jacobi metric, that is nicely tailored to a probability distribution $\pi$ on $\mcX$.

\section{Hamiltonian Mechanics} \label{sec:hamiltonian}

Physicists \cite{Duane1987216} proposed  a MC sampling scheme that uses Hamiltonian dynamics to improve convergence rates. The method  mimics the movement of a body under potential and kinetic energy changes to avoid diffusive behavior. The stationary probability will be linked to the potential energy. The reader is invited to read \cite{Neal11} for an enlightening survey of the subject.

The setup is as follows: let $\mcX$ be a manifold, and let $\pi$ be a target distribution on $\mcX$. As with the Metropolis-Hastings algorithm, we start at some point $q_0\in \mcX$. However, we use an analogue of the laws of physics to tell us where to go for future steps. In this section, we will work on Euclidean spaces $\mcX=\RR^d$ with the standard Euclidean metric, but in the next section we will use a special metric induced by $\pi$.


The model incorporates two types of energy: potential energy and kinetic energy. The potential energy is a function solely of the position of a particle, whereas the kinetic energy depends not just on the position but also its motion, and in particular its momentum; both the position and the momentum are elements of $\RR^d$. 
In a more abstract setting, we can define a potential energy function $V:\RR^d\to\RR$ and a kinetic energy function $K:\RR^d\times\RR^d\to\RR$. Both $V$ and $K$ should be smooth functions. We typically write a point in $\RR^d\times\RR^d$ as $(q,p)$, where $q,p\in \RR^d$. We call $q$ the \emph{position} and $p$ the \emph{momentum}. We will sometimes write $\RR^d_{\pos}$ for the space of positions and $\RR^d_{\vel}$ for the space of momenta to avoid confusion.

The position and momentum play different roles. The position space is the state space. The momentum, on the other hand, is only an auxiliary variable which helps us update the position and is of no interest in its own right. 

We define the \emph{Hamiltonian function} $H:\RR^d_{\pos}\times\RR^d_{\vel}\to\RR$ by $H(q,p)=V(q)+K(q,p)$. This represents the total energy of a particle with position $q$ and momentum $p$.

According to the laws of Hamiltonian mechanics, as a particle with position $q(t)$ and momentum $p(t)$ travels, $q$ and $p$ satisfy the Hamilton equations \begin{equation} \label{hamilton} \frac{dq}{dt}=\frac{\partial H}{\partial p},\qquad \frac{dp}{dt}=-\frac{\partial H}{\partial q},\end{equation} where if $d>1$ this means that these equations hold in each coordinate, i.e.\ \[\frac{dq_i}{dt}=\frac{\partial H}{\partial p_i},\qquad \frac{dp_i}{dt}=-\frac{\partial H}{\partial q_i}.\] 

A simple consequence of the Hamilton equations is the following  proposition (see~\cite[\S2.2]{Neal11}).

\begin{prop}[Conservation of energy] \label{consenergy} The Hamiltonian is constant along any trajectory $(q(t),p(t))$ satisfying (\ref{hamilton}). \end{prop}

The Hamilton equations (\ref{hamilton}) tell us how the position and momentum of a particle evolve over time, given a starting position and momentum. As a result, we can compute $q(t_1)$ and $p(t_1)$, its position and momentum at time $t_1$. Here, $t_1$ is a parameter that can be tuned to suit the distribution.

Running Hamiltonian Monte Carlo is similar to the classical MCMC in that we propose a new point in $\mcX$ based on the current point, and then either accept or reject it. However, the method of proposal differs from the classical method. In Hamiltonian Monte Carlo, to choose $q_{i+1}$ from $q_i$, we select a momentum vector $p_i$ from $\RR^d_{\vel}$, chosen according to a $\mcN(0,I_d)$ distribution. We then solve the Hamilton equations (\ref{hamilton}) with initial point $q(0)=q_i$ and $p(0)=p_i$, and we let $q_{i+1}^\ast=q(t_1)$ and $p_{i+1}^\ast=p(t_1)$. We accept and make $q_{i+1}=q_{i+1}^\ast$ with probability \[\alpha=\min(1,\exp(-H(q_{i+1}^\ast,p_{i+1}^\ast)+H(q_i,p_i)))\] and reject and let $q_{i+1}=q_i$ otherwise. However, by Proposition~\ref{consenergy}, the exponential term is 1, i.e.\ theoretically, we should always accept.

\begin{rem} In practice, we may occasionally reject, as it will be necessary to solve the Hamilton equations numerically, thereby introducing numerical errors. 
In truly high dimensional settings,  the acceptance probability should be tuned to around $0.65$ \cite{Beskos11} by varying the step size in the numerical integration procedure.
\end{rem} 

Note that at every step, we pick a fresh momentum vector, independent of the previous one.
In order to make the stationary distribution of the $q_n$'s be $\pi$, we choose $V$ and $K$ following Neal \cite{Neal11}; we take \begin{equation} \label{VKdef} V(q) = -\log\pi(q)+C,\qquad K(p)=\frac{1}{2}\|p\|^2,\end{equation} where $C$ is a convenient constant. Note that $V$ only depends on $q$ and $K$ only depends on $p$. $V$ is larger when $\pi$ is smaller, and so 
trajectories are able to move more quickly starting from lower density regions than out of higher density regions.


\begin{rem} The Hamiltonian is \emph{separable} in the case of a distribution on Euclidean space, meaning that it can be expressed as the sum of a function of $q$ alone and a function of $p$ alone. However, we write $K(q,p)$ as a function of both $q$ and $p$, because tangent vectors should not be thought of as being detachable from the underlying manifold. \end{rem}

\begin{rem} It is possible to choose other distributions for the momentum. Doing so changes $K$ accordingly. See \cite[\S 3.1]{Neal11} for a discussion of how to relate $K$ and the distribution of $p$. 
\end{rem}

\begin{ex} \label{firstnormalex} If $\pi=\mcN(0,\Sigma)$ is a multivariate Gaussian distribution, then, by choosing $C$ to be a suitable normalizing constant, we can take \[V(q)=\frac{1}{2}q\tp\Sigma^{-1}q, \qquad K(p)=\frac{1}{2}\|p\|^2.\]
Taking $\pi=\mcN(0,1)$ to be the standard univariate normal, we can see the trajectory of HMC explicitly. Suppose we choose a terrible starting point $q_0=1000$ and $p_0=1$, so that we are initially moving away from the high-density region. We quickly recover: the Hamilton equations become $\frac{dq}{dt}=p,\, \frac{dp}{dt}=-q$. Solving these equations with our initial conditions, we find that $q(t)=1000\cos(t)+\sin(t)$. Suppose we take $t_1=1$, meaning that we follow the trajectory until $t=1$ before choosing a new $p$. Then when $t=t_1=1$, we have $q(1)=1000\cos(1)+\sin(1)\approx 541$. Hence, in only one step, we have already made a substantial recovery. On the other hand, if we start at $q_0=1.5$, again with $p=1$, then after one second, we reach $q(1)=1.5\cos(1)+\sin(1)\approx 1.65$, so we stay in a sensible location. \end{ex}

There are several reasons to expect Hamiltonian Monte Carlo to perform better than classical MCMC. For one thing, there are no (or at least fewer) rejections. Also, since the potential energy is greater in regions of $\mcX$ with lower $\pi$-density, the Hamiltonian trajectory moves more quickly starting at such a point, allowing a rapid escape; on the other hand, the trajectory moves more slowly away from a region of high density, encouraging the chain to spend more time there. The reason for this behavior is that the potential energy is higher in regions of lower density, which makes the total energy higher, which in turn will eventually make the particle move faster since it will transfer its energy from potential to kinetic energy. Furthermore, an unfortunate tendency of classical MCMC is diffusive behavior: the chain moves back and forth for several steps in a row; this behavior is less likely in the HMC setting, since the potential energy dictating the step size changes at every step. Finally, we expect it to perform better because the Hamiltonian trajectory adjusts continuously to the local shape of $\mcX$, rather than taking discrete steps that may not detect the fine structure of the space.

In practice, we have found that HMC outperforms MCMC as did Neal in \cite[\S 3.3]{Neal11} who performed simulations that demonstrated HMC's advantage over MCMC.


\section{Curvature} \label{sec:geometry}

We can associate a notion of curvature to a Markov chain, an idea introduced by Ollivier in \cite{Ollivier09} and Joulin in \cite{Joulin2007}
following work of Sturm~\cite{Sturm,Sturm2}. We apply this notion of curvature to the HMC chain whose stationary distribution is our target distribution. This will allow us  to obtain error bounds for numerical integration
in \S\ref{sec:examples} using Hamiltonian Monte Carlo in the cases when HMC has positive curvature. 

In order to bring the geometry and the probability closer together, we will deform our state space $\mcX$ 
to take the probability distribution into account, in a manner reminiscent of the inverse transform method mentioned in the introduction. 
Formally, this amounts to putting a suitable \emph{Riemannian metric} on our $\mcX$.
 
Here $\mcX$ is a \emph{Riemannian manifold}: the Euclidean space $\RR^d$ with the extra Riemannian metric.
Given a probability distribution $\pi$ on $\mcX=\RR^d$, we now define a metric on
$\mcX$ 
that is tailored to $\pi$ and the Hamiltonian it induces (see \S\ref{sec:hamiltonian}). This construction is originally due to Jacobi, but our treatment follows Pin in \cite{Pin75}.

\begin{defn} Let $(\mcX,\langle\cdot,\cdot\rangle)$ be a Riemannian manifold, and let $\pi$ be a probability distribution on $\mcX$. Let $V$ be the potential energy function associated to $\pi$ by (\ref{VKdef}). For $h\in\RR$, we define the Jacobi metric to be \[g_h(\cdot,\cdot)=2(h-V)\langle\cdot,\cdot\rangle.\] \end{defn}

\begin{rem} 
$(\mcX,g_h)$ is not necessarily a Riemannian manifold, since $g_h$ will not be positive definite if $h-V$ is ever nonpositive. We could remedy this situation by restricting to the subset of $\mcX$ on which $h-V>0$. In fact, this restriction will happen automatically, as we will always select values of $h$ for which $h-V>0$; indeed, $h$ will be $V+K$, and if $d\ge 2$, $K$ will be positive almost surely. \end{rem}

The point of the Jacobi metric is the following result of Jacobi, following Maupertuis:

\begin{thm}[Jacobi-Maupertuis Principle, \cite{Jacobi09}] Trajectories $q(t)$ of the Hamiltonian equations \ref{hamilton} with total energy $h$ are geodesics of $\mcX$ with the Jacobi metric $g_h$.\end{thm}

This theorem provides a link between a probability distribution $\pi$ and the Hamiltonian system from HMC.
The Riemannian manifold equipped with the Jacobi metric encodes both the behavior of HMC and the target distribution $\pi$. 
This allows us to reduce the analysis of HMC to a geometric problem by studying the geodesics of this Riemannian manifold through the usual geometric tools such as curvature. This follows the spirit of comparison theorems in classical Riemannian geometry where complicated geometries are compared to the three types of spaces --- the negatively curved spaces, where geodesics starting at the same point spread out; flat spaces, where geodesics correspond to straight lines; and positively curved spaces, where geodesics starting at the same point meet again. We will show that, in many cases, the manifolds associated to HMC are close to spheres in high dimensions, and that an HMC random walk reduces to something close to a geodesic random walk on a sphere (where geodesics are great circles on the sphere).

The most convenient way for us to think about the Jacobi metric on $\mcX$ is as distorting the space to suit the probability measure. In order to do this, we make regions of high density larger, and we make regions of low density smaller. However, the Jacobi metric does not completely override the old notion of distance and scale; the Jacobi metric provides a \emph{compromise} between physical distance and density of the probability measure.

Another, essentially equivalent, way to think about the Jacobi metric is as a distortion of time. This is particularly natural since Hamiltonians describe how states of a system evolve over time. In this analogy, the Jacobi metric slows down time in regions of high probability and speeds it up in regions of low probability.  As a result it takes a long time to move from high to low probability regions, but less time to move in the opposite direction.

As the Hamiltonian Monte Carlo progresses, $h$ changes at every step, so the metric structure varies as we run the chain, moving between different Riemannian manifolds. However, we prefer to think of the chain as running on a single manifold, with a changing metric structure.

Another important notion for us is that of the exponential map. Given a Riemannian manifold $\mcX$ and a point $x\in \mcX$, there is a canonical map $\exp:T_x\mcX\to \mcX$. If $v\in T_x\mcX$, then $\exp(v)$ is obtained by following the unique geodesic in the direction of $v$ whose distance is $\|v\|$ measured in Riemannian metric; $\exp(v)$ is then the endpoint of this geodesic.

Here we provide 
some facts that give an intuition about sectional curvature.

The \emph{sectional curvature} in the plane spanned by two linearly independent tangent vectors $u,v\in T_x\mcX$ is defined for
a $d$-dimensional Riemannian manifold $\mcX$. Two distinct points $x,y \in \mcX$ and two tangent vectors $v \in T_x\mcX,v' \in T_y\mcX$ at $x$ and $y$ are related to each other by parallel transport along the geodesic in the direction of $u$: $v'$ is a tangent vector at $y$ which is, in a suitable sense, parallel to $v$ and is obtained from $v$ by constructing a family of parallel vectors at each point along the geodesic from $x$ to $y$. Let $\delta$ be the length of the geodesic between $x$ and $y$, and $\eps$ the length of $v$ (same as $v'$). The sectional curvature $\Sec_x(u,v)$ at point $x$ is defined in terms of the geodesic distance $\rho$ between the two endpoints $\exp_x(\eps v)$ and $\exp_y(\eps v')$ as the quantity that satisfies the equation
\[ \rho(\exp_x(\eps v),\exp_y(\eps v')) = \delta \left( 1-\frac{\eps^2}{2} \Sec_x(u,v) + O(\eps^3 + \eps^2 \delta) \right) \mbox{ as }(\eps,\delta) \to 0.\]
Figure \ref{fig:SectionalCurvatureSphere} depicts the sectional curvature on a sphere, where the endpoints of the two red curves starting at $x$ and $y$ correspond to the endpoints $\exp_x(\eps v)$ and $\exp_y(\eps v')$.
The dashed lines indicate geodesics going from $x$ to $y$, and from $\exp_x(\eps v)$ to $\exp_y(\eps v')$. As we get closer to the north pole, endpoints get closer together. Sectional curvature describes this convergence of endpoints: higher curvature means faster convergence. 
See also~\cite[Proposition 6]{Ollivier09}.
\begin{figure}
    \centering
    \includegraphics[width=0.35\textwidth]{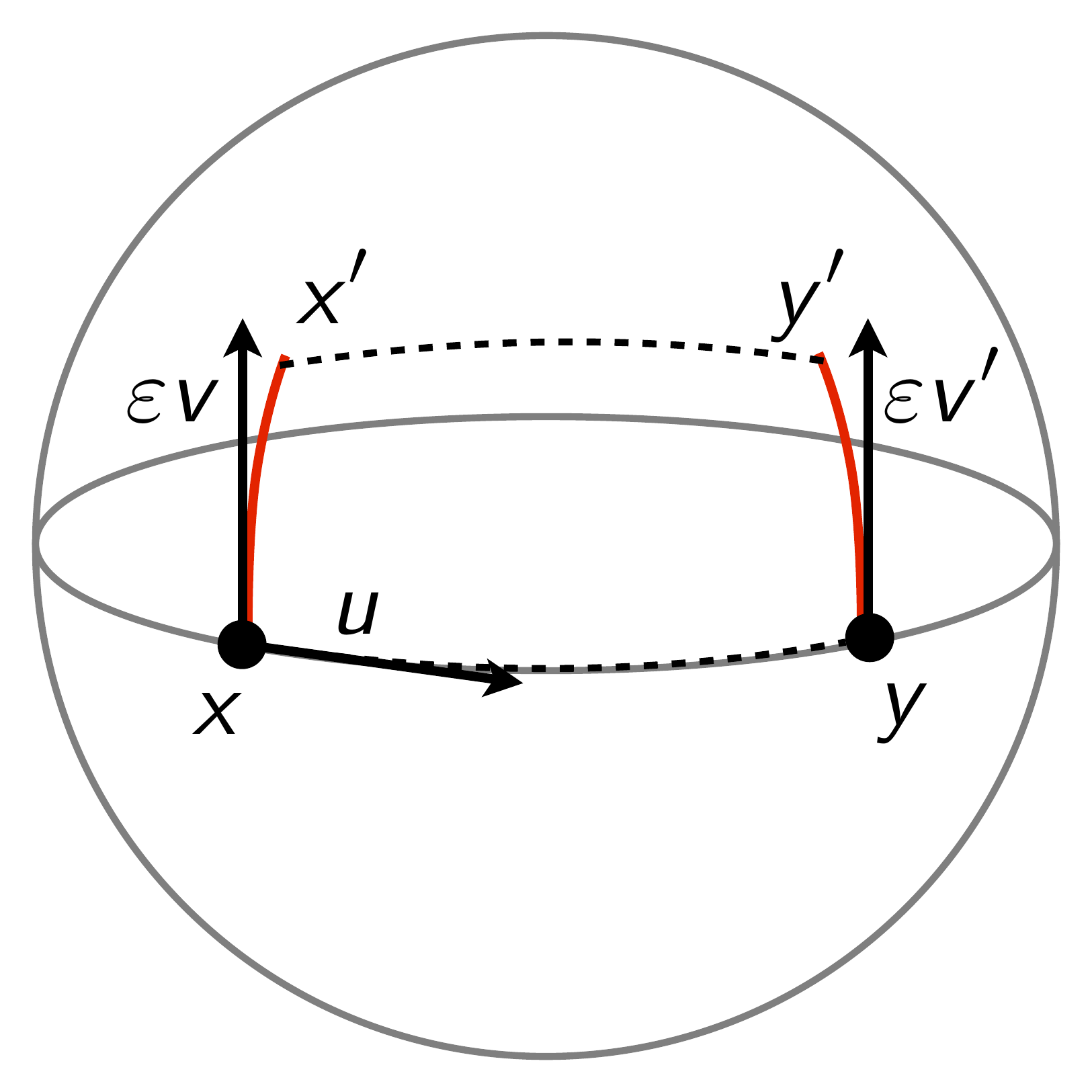}
    \includegraphics[width=0.35\textwidth]{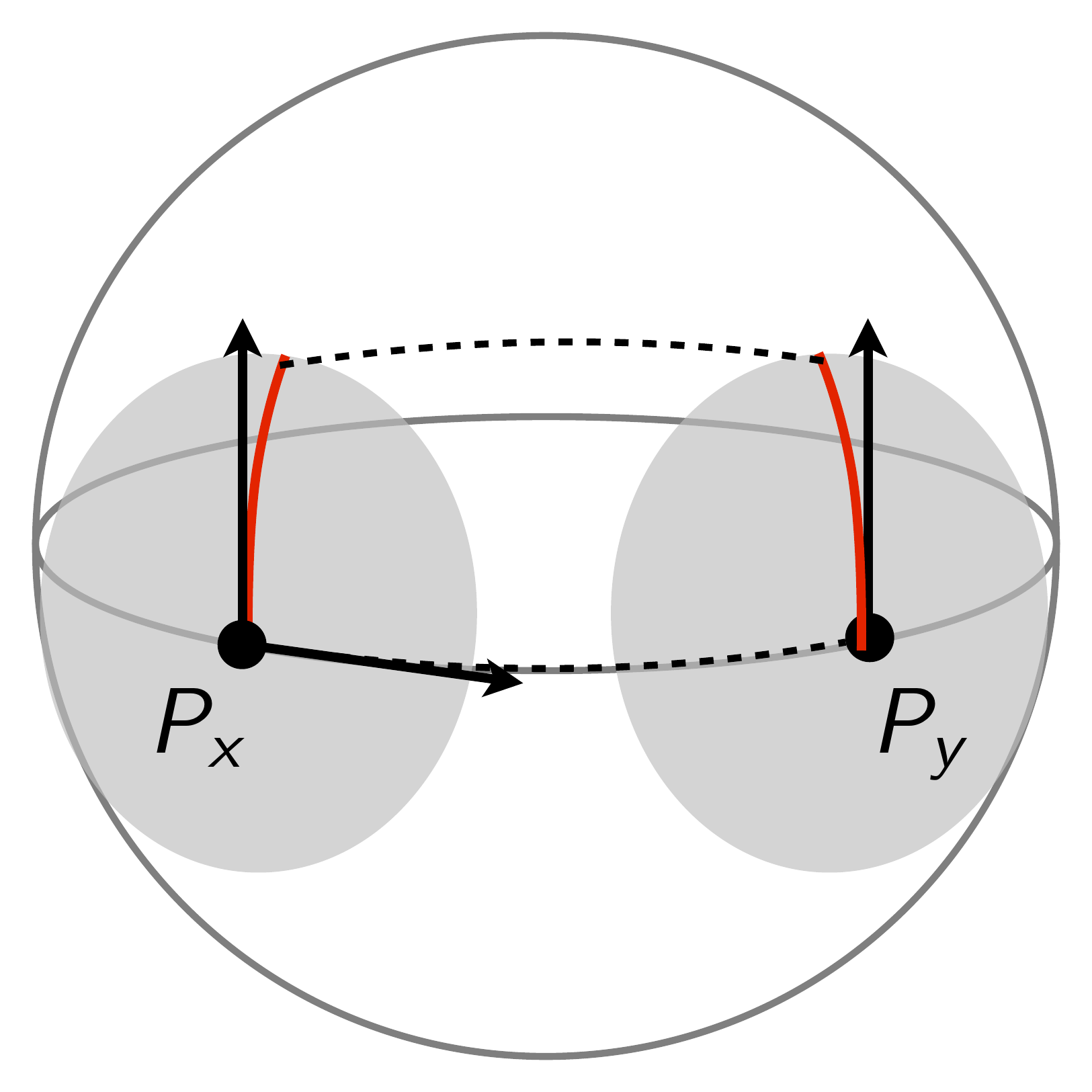}
    \caption{Sketch of positive sectional curvature (left) and coarse Ricci curvature (right) on a sphere.}
    \label{fig:SectionalCurvatureSphere}
\end{figure}

We let $\InfSec$ denote the infimum of $\Sec_x(u,v)$, where $x$ runs over $\mcX$ and $u,v$ run over all pairs of linearly independent tangent vectors at $x$.

\begin{rem} \label{riccicheating} In practice, it may not be easy to compute $\InfSec$ precisely. As a result, we can approximate it by running a suitable Markov chain on the collection of pairs of linearly independent tangent vectors of $\mcX$; say we reach states $(x_1,u_1,v_1),(x_2,u_2,v_2),\ldots,(x_t,u_t,v_t)$. Then we can approximate $\InfSec$ by the \emph{empirical} infimum of the sectional curvatures $\inf_{1\le i\le t}\Sec_{x_i}(u_i,v_i)$. This approach has computational benefits, but also theoretical benefits: it allows us to ignore low sectional curvatures that are unlikely to arise in practice. \end{rem}

Note that $\Sec$ depends on the metric. There is a formula, due to Pin \cite{Pin75}, connecting the sectional curvature of a Riemannian manifold equipped with some reference metric, with that of the Jacobi metric. We write down an expression for the sectional curvature in the special case where $\mcX$ is a Euclidean space and $u$ and $v$ are orthonormal tangent vectors at a point $x\in\mcX$:
\begin{multline} \label{secformula}
\Sec(u,v) = \frac{1}{8(h-V)^3} \Big( 2(h-V) \Big[ \langle (\Hess V)u,u \rangle + \langle (\Hess V)v,v \rangle \Big] \\ 
+ 3 \Big[ \|\grad V\|^2\cos^2(\theta) + \|\grad V\|^2\cos^2(\beta)\Big] - \|\grad V\|^2 \Big).
\end{multline}
Here, $\theta$ is defined as the angle between $\grad V$ and $u$, and $\beta$ as the angle between $\grad V$ and $v$, in the standard Euclidean metric.
We will also need a special notion of curvature (Figure~\ref{fig:SectionalCurvatureSphere}), known as \emph{coarse Ricci curvature}, for Markov chains. 
To define it, we will use the distance between measures to be the standard Wasserstein metric.

\begin{defn} Let $\mcX$ be a metric measure space with metric $\rho$, and let $\mu$ and $\nu$ be two probability measures on $\mcX$. Then the Wasserstein distance
between them is defined as \[W_1(\mu,\nu)=\inf_{\xi\in\Pi(\mu,\nu)} \iint_{\mcX\times\mcX} \rho(x,y)\, \xi(dx,dy).\] Here $\Pi(\mu,\nu)$ is the set of measures on $\mcX\times\mcX$ whose marginals are $\mu$ and $\nu$. \end{defn}

If $P$ is the transition kernel for a Markov chain on a metric space $(\mcX,\rho)$, let $P_x$ denote the transition probabilities starting from state $x$. We define the coarse Ricci curvature $\kappa(x,y)$ as the function that verifies:
\[W_1(P_x,P_y) = (1-\kappa(x,y))\rho(x,y).\] We write $\kappa$ for $\inf_{x,y\in\mcX} \kappa(x,y)$. 

We shall see in \S\ref{sec:examples} that there is a close connection between sectional curvature of a Riemannian manifold and coarse Ricci curvature of a Markov chain.




\subsection{Positive Curvature} \label{sec:poscurv}

In order to produce error bounds for a distribution $\pi$, it is necessary for the HMC process associated to $\pi$ to have positive curvature. Thus, it is important to know, at an intuitive level, when to expect this to happen, and when to expect this to fail.

Roughly, coarse Ricci curvature for a Markov chain on a metric space $\mcX$ can be interpreted as follows: Suppose $x,y\in\mcX$ are two nearby points. Suppose we take a step starting from $x$ to a point $x'$, and we take the ``corresponding'' step from $y$ to $y'$.  If the coarse Ricci curvature $\kappa(x,y)$ is positive, then on average, the distance between $x'$ and $y'$ is \emph{smaller} than the distance between $x$ and $y$. By contrast, if the curvature is 0, then the distance between $x'$ and $y'$ is on average \emph{equal} to the distance between $x$ and $y$, whereas if the curvature is negative, then the distance between $x'$ and $y'$ is on average \emph{greater} than the distance between $x$ and $y$.

Based on this interpretation, we expect multimodal distributions $\pi$ to give us negative curvature. To see this, suppose $\pi$ is a symmetric bimodal distribution with modes $a$ and $b$, and let $x$ and $y$ be two nearby points between the two modes, with $x$ slightly closer to $a$ than $b$, and $y$ slightly closer to $b$ than $a$. Then, if we take a step from $x$, the resulting point $x'$ is likely to move toward $a$, whereas if we take a step from $y$, the resulting point $y'$ is likely to move toward $b$. Hence, we expect the distance between $x'$ and $y'$ to be larger than the distance between $x$ and $y$, giving us negative curvature; see Figure \ref{fig:BimodalNegativeCurvature} for an illustration.
\begin{figure}
    \centering
    \includegraphics[width=0.6\textwidth]{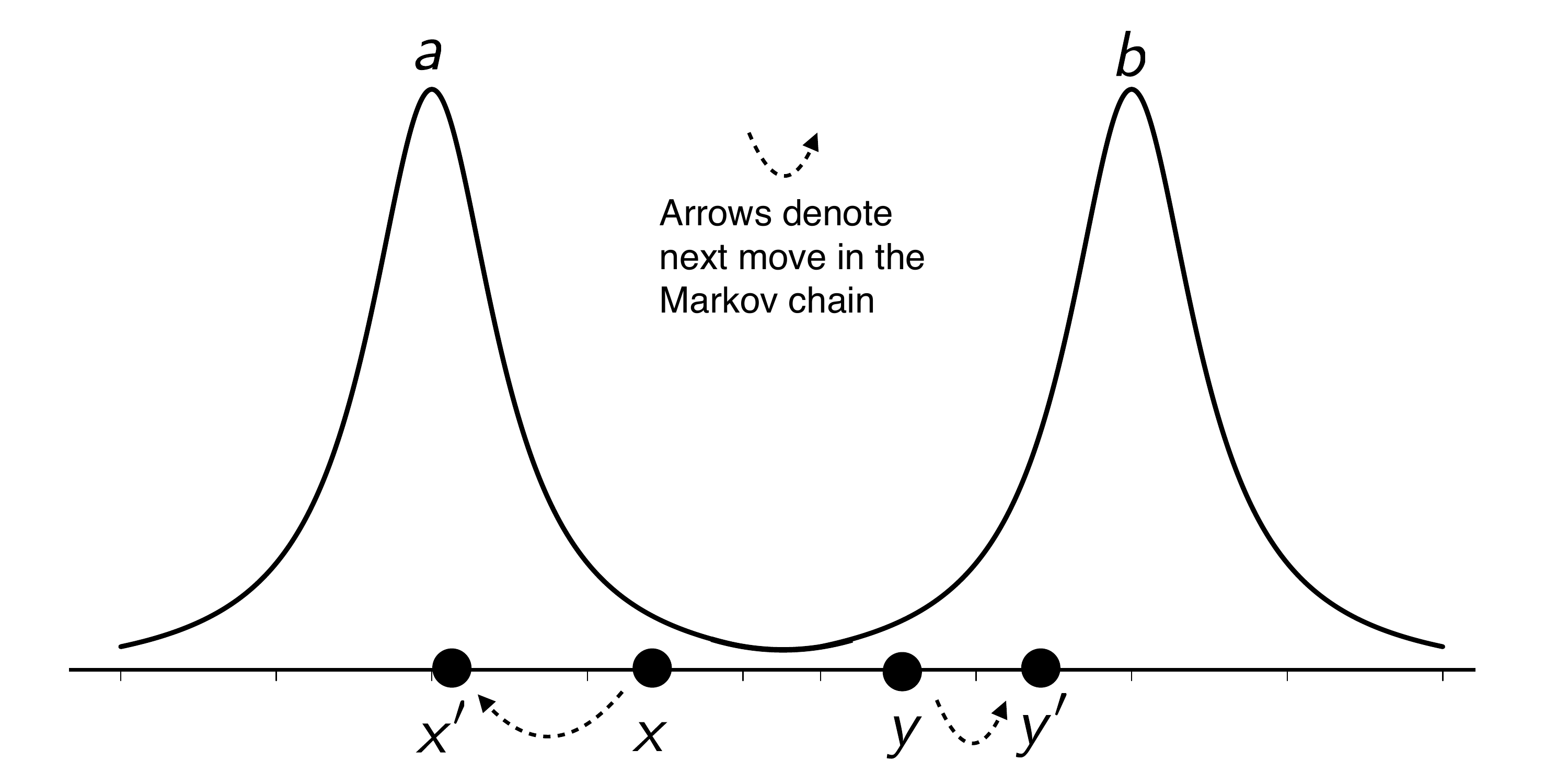}
    \caption{Sketch of negative curvature in a bimodal distribution. Two Markov chains at $x$ and $y$ move away from each other, on average, to higher density parts of the space.}
    \label{fig:BimodalNegativeCurvature}
\end{figure}

By contrast, unimodal distributions frequently have positive curvature, since points want to move closer to the mode, as we saw in Example~\ref{firstnormalex}. 

\section{Concentration Inequalities} \label{sec:concentration}

Now that we have introduced all the necessary ingredients, we review the concentration results of Joulin and Ollivier and apply them to the setting of Hamiltonian Monte Carlo. 

From \cite{Ollivier09}, we use the following definitions. These apply to general metric measure spaces and not specifically to HMC.

\begin{defn} \label{lipschitz} The \emph{Lipschitz norm} (or, more precisely, seminorm) of a function $f:(\mcX,\rho)\to\RR$ is
\[ \|f\|_{\Lip} := \sup_{x,y \in\RR^d } \frac{|f(x)-f(y)|}{\rho(x,y)}.\] If $\|f\|_{\Lip}\le C$, we say that $f$ is $C$-Lipschitz. \end{defn}

\begin{defn} The \emph{coarse diffusion constant} of a Markov chain on a metric space $(\mcX,\rho)$ with kernel $P$ at a state $q\in \mcX$ is the quantity \[ \sigma(q)^2 := \frac{1}{2} \iint_{\mcX\times \mcX} \rho(x,y)^{2} \, P_{q}(dx)\, P_{q}(dy).\] \end{defn}

The coarse diffusion constant controls the size of the steps at a point $q\in \mcX$.

\begin{defn} The \emph{local dimension} $n_q$ is \[ n_{q} := \inf_{\substack{f:\mcX\to\RR \\ \text{$f$ 1-Lipschitz}}} \frac{\iint_{\mcX\times\mcX} \rho(x,y)^{2} \, P_q(dx)\, P_q(dy)}{\iint_{\mcX\times\mcX} |f(x) - f(y)|^{2} \, P_q(dx)\, P_q(dy)}. \]
\end{defn}


\begin{defn} The \emph{granularity} $\sigma_\infty$ is \[\sigma_\infty=\frac{1}{2}\sup_{x\in\mcX} \diam P_x.\] \end{defn}

We record the values for these and other expressions that show up in the concentration inequalities in Table~\ref{locdimtable} on page~\pageref{locdimtable}, in the case of a multivariate Gaussian distribution.




We now state Joulin and Ollivier's error bound. We assume that the $x_i$'s are chosen by running a Markov chain (not necessarily HMC) with stationary distribution $\pi$ on a metric space $\mcX$, and that the coarse Ricci curvature $\kappa$ is \emph{positive}.



\begin{thm}[\cite{JO10}] \label{thm:ConcentrationInequality} 
Let \[V^2(\kappa,T) = \frac{1}{\kappa T}\left(1+\frac{T_0}{T}\right)\sup_{x\in\mcX} \frac{\sigma(x)^2}{n_x\kappa}.\] 
Then, assuming that $\sigma_\infty<\infty$, we have \begin{equation}\PP_x(|\widehat{I}-\EE_x\widehat{I}|\ge r\|f\|_{\Lip}) \le \begin{cases}  2e^{-r^2/(16V^2(\kappa,T))} & 0<r<\frac{4V^2(\kappa,T)\kappa T}{3\sigma_\infty}, \\ 2e^{-\kappa Tr/(12\sigma_\infty)} & r\ge\frac{4V^2(\kappa,T)\kappa T}{3\sigma_\infty}. \end{cases} \label{concineq} \end{equation} \end{thm}

\begin{rem} The appearance of $\kappa$ in these expressions has an elegant interpretation. As we run a Markov chain, the samples drawn are not independent. The curvature $\kappa$ can be thought of as a measure of the deviation from independence, so that $1/\kappa$ samples drawn from running the chain substitute for one independent sample. \end{rem}

In order to use Theorem~\ref{thm:ConcentrationInequality} in the case of HMC, we must say something about the symbols that appear in~(\ref{concineq}), and what they mean in our context. The only ones that pose any serious difficulties are $\|f\|_{\Lip}$ and $\sigma_\infty$, both of which ought to be computed in the Jacobi metric. However, we have a different Jacobi metric for each total energy $h$, and the only requirement on $h$ is that it be at least as large as $V$. In the case of the Lipschitz norm, this would suggest that we define the Lipschitz norm for HMC to be the supremum of the Lipschitz quotient over all pairs of points and $h\ge V$. This approach will not be successful, however, as it would require division by $h-V$, which can be made arbitrarily small.

Instead, we make use of high dimensionality, and recall that the momentum is distributed according to $\mcN(0,I_d)$, and that $K=\frac{1}{2}\|p\|^2$, so that the distance $\rho_h(x,y)$ between two very close points $x$ and $y$ in the Jacobi metric $g_h$ is $\|p\|\|x-y\|+O(\|x-y\|^2)$. Hence the Lipschitz quotient in the Jacobi metric between two nearby points is \[\frac{|f(x)-f(y)|}{\rho_h(x,y)}=\frac{|f(x)-f(y)|}{\|p\|\|x-y\|+O(\|x-y\|^2)},\] or the standard Lipschitz norm divided by $\|p\|$, up to higher order terms. The random variable $\|p\|^2$ has a $\chi^2$ distribution. Using tail bounds from[Lemma 1, p.\ 1325]~\cite{LM00} (which will also be used several times in~\S\ref{sec:GaussianExample}), \[\PP\left(\Big|\|p\|^2-d\Big|\ge d^{3/4}\right)=O\left(e^{-c\sqrt{d}}\right)\] for a suitable constant $c>0$. Assuming we stay away from this exceptional set, we have $\|p\|^2\approx d$, and the Lipschitz norm of a function $f$ with respect to the Jacobi metric is the Euclidean Lipschitz norm multiplied by $\frac{1}{\sqrt{d}+O(d^{1/4})}$. For sufficiently large $d$, away from the exceptional set, we take the Lipschitz norm with respect to the Jacobi metric to be $\frac{2}{\sqrt{d}}$ times the Euclidean Lipschitz norm. Here and elsewhere, we have taken a \emph{probabilistic} approach to computing the expressions that show up in Theorem~\ref{thm:ConcentrationInequality}. In this case, we estimate the Lipschitz norm by seeing what the Lipschitz quotient looks like for ``typical'' pairs of points.


Similarly, for $\sigma_\infty$, we interpret the diameter of $P_x$ to mean the farthest distance between two ``typical'' points in the support of $P_x$, where the distance is calculated in the Jacobi metric, interpreted as in the previous paragraph. In the next section, we will continue in this spirit.

Intuitively, the reason this approach is appropriate is that events that occur with such vanishingly small probability that they are unlikely to occur at any step in the walk cannot meaningfully impact the mixing time of the chain. More precisely, in our case, we will be running HMC for polynomially many (in $d$) steps, and the exceptional events occur with probability $O(e^{-c\sqrt{d}})$ for some constant $c>0$. Hence, for $d$ sufficiently large, it is extremely unlikely that we will ever observe an exceptional event throughout the HMC walk. Another essentially equivalent approach is to use a coupling argument to couple HMC with a modified version that truncates the distribution on initial momenta so as to disallow initial momenta for which $K(q,p)$ is of an anomalous magnitude. This approach would lead to the same results as we obtain below.

\section{Examples} \label{sec:examples}

Here we show how curvature can quantify running time $T$ in three examples: the multivariate Gaussian distribution, the multivariate $t$ distribution, and Bayesian image registration. In the Gaussian case, we are able to prove that there is a computable positive number $\kappa_{d,\Lambda}$, depending on the dimension of the space and the covariance matrix, so that with very high probability, the coarse Ricci curvature is at least $\kappa_{d,\Lambda}$; furthermore, we explicitly bound the probability of the exceptional set. In the other two cases, it is not possible to work analytically, so instead we take an empirical approach. In all our cases, we observe that there is a high concentration of positive (and nearly constant) curvature. We give evidence that empirical curvature is an interesting diagnostic tool to assess the convergence of HMC in practice.

In the following, we will show how to choose the starting point $x_0$ so that our analysis starts with a chain close to stationarity (so that only a small $T_0$ is required). For instance, we will start the chain at the mode for our Gaussian and t distribution examples. For our Bayesian image registration example, we will start the chain at the mode of the prior distribution. In the Bayesian setting, $\pi$ will always be absolutely continuous with respect to the prior.
Finding the right Bayesian model and $\pi_0$ for image registration is the focus of \cite{SRH14}. In the terminology of Lov\'asz~\cite{Lovasz99}, we are using a \emph{warm start}.

The question of how long to run the Markov chain (quantifying $T$) is important  in practice as every step comes at considerable computational cost, both in terms of memory storage and  runtime. For instance, for a real three dimensional medical imaging problem with 100 subjects, we need around 100 MB per HMC step. This can quickly exhaust the storage capabilities of a computing cluster. 

\subsection{Multivariate Gaussian Distribution} \label{sec:GaussianExample}

In this section, we estimate the coarse Ricci curvature of HMC with Gaussian target distribution in high dimensions, with some mild hypotheses on the covariance matrix, in order to obtain error bounds from the Joulin-Ollivier machinery. An overview of our approach is as follows: first, we show that the sectional curvature of Jacobi metrics coming from Gaussian target distributions are positive and concentrated around their means in high dimensions. These spaces equipped with Jacobi metrics are close to spheres, and we bound the coarse Ricci curvature for certain Markov chains on spheres. Finally, we show that the coarse Ricci curvature on spheres is close to that of HMC with Gaussian target distribution.

\begin{figure}
    \centering
    \includegraphics[width=0.5\textwidth]{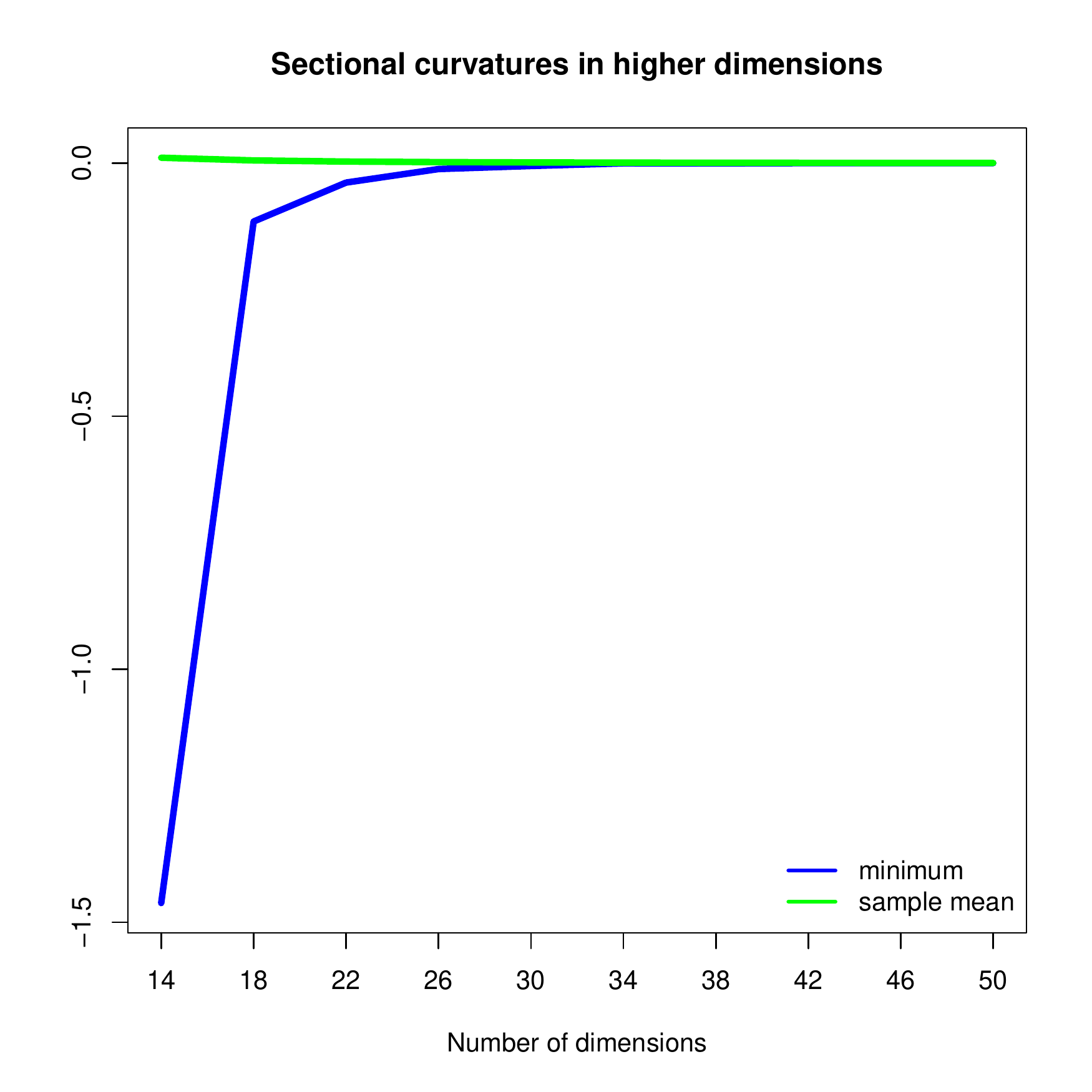}
    \caption{(Identity covariance structure) Minimum and sample average of sectional curvatures for $14,18,\dots,50$-dimensional multivariate Gaussian $\pi$ with identity covariance. For each dimension, we run a HMC random walk with $T=10^4$ steps, and at each step, we compute sectional curvatures for $1000$ uniformly sampled orthonormal 2-frames in $T_q\X$ (see Remark~\ref{rem:stiefel}).}
    \label{fig:CurvatureEvolution}
\end{figure}
\begin{figure}
    \centering
    \includegraphics[width=0.32\textwidth]{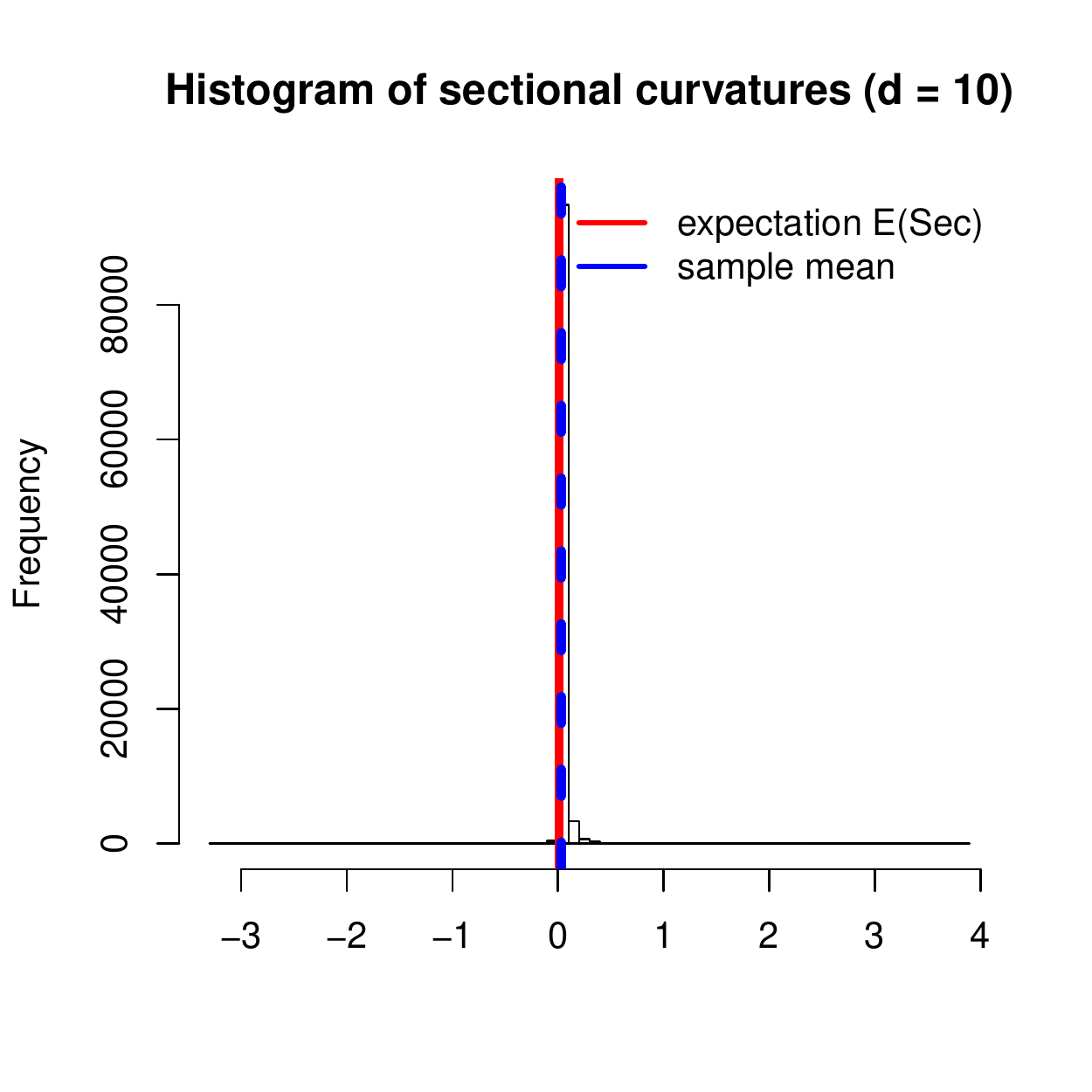}
    \includegraphics[width=0.32\textwidth]{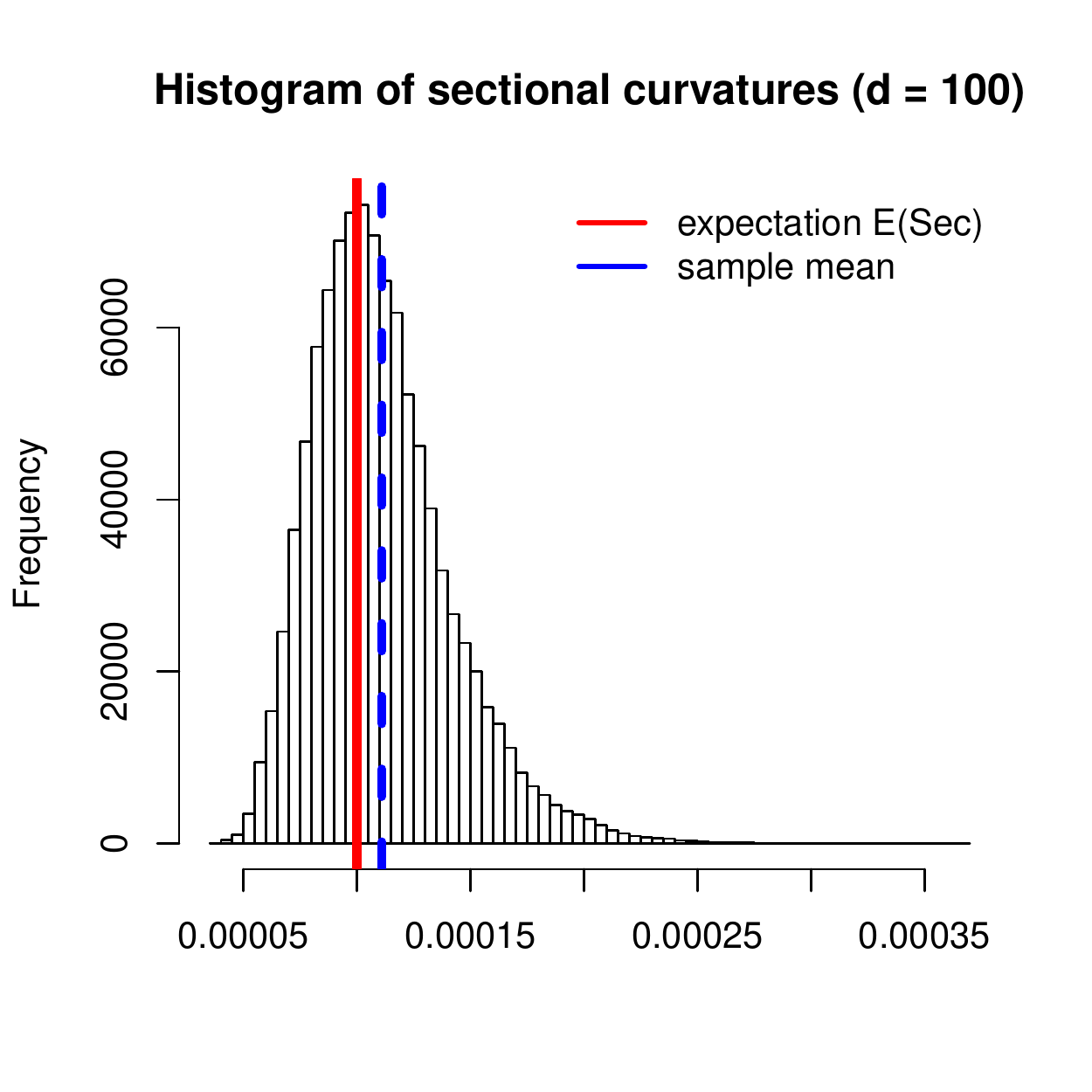}
    \includegraphics[width=0.32\textwidth]{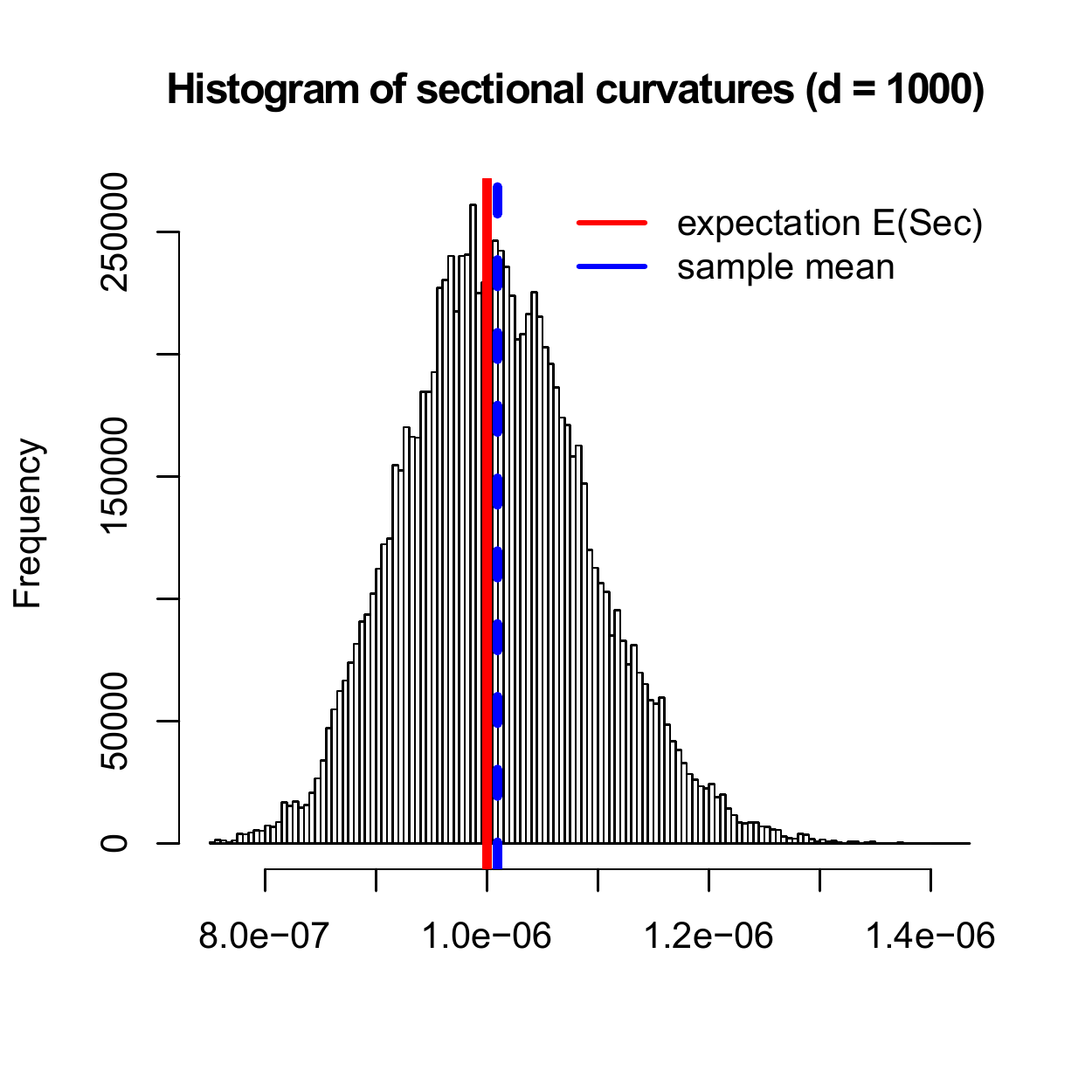}
    \caption{(Identity covariance structure) HMC after $T=10^4$ steps for multivariate Gaussian $\pi$ with identity covariance in $d=10,100,1000$ dimensions. At each step we compute the sectional curvature for $d$ uniformly sampled orthonormal 2-frames in $T_q\X$ (see Remark \ref{rem:stiefel}).}
    \label{fig:IdentityCovarianceDimensions}
\end{figure}
The first step is to apply Theorem \ref{thm:ConcentrationInequality} to the case of multivariate Gaussian distributions, running an HMC Markov chain to sample from $\pi=\mcN(0,\Sigma)$. We prove that, with high probability, the lower bound on sectional curvature is positive.
Empirically, we can see in Figure~\ref{fig:IdentityCovarianceDimensions} that the distribution of sectional curvatures approaches a Gaussian, which does imply that sectional curvatures are
in fact positive. Furthermore, Figure~\ref{fig:CurvatureEvolution} shows that the minimum and mean sectional curvatures during the HMC random walk tend closer with increasing dimensionality and meet visually at around 30 dimensions. More formally, we now show that sectional curvature is indeed positive with high probability in high dimensions; furthermore, the proof shows that, with high probability, the sectional curvature is very close to the mean. 

\begin{lem} \label{lem:SecCurvature}
Let $C\ge 1$ be a universal constant and $\pi$ be the $d$-dimensional multivariate Gaussian $\mcN(0,\Sigma)$, where $\Sigma$ is a $(d \times d)$ covariance matrix, all of whose eigenvalues lie in the range $[1/C, C]$. We denote by $\Lambda=\Sigma^{-1}$  the precision matrix. Let $q$ be distributed according to $\pi$, and $p$ according to a Gaussian $\mcN(0,I_d)$. Further, $h = V(q) + K(q,p)$ is the sum of the potential and the kinetic energy. The Euclidean state space $\mcX$ is equipped with the Jacobi metric $g_h$. Pick two orthonormal tangent vectors $u,v$ in the tangent space  $T_q\mcX$ at point $q$. Then the sectional curvature $\Sec$ from expression (\ref{secformula}) is a random variable bounded from below with probability
\[\Prob(d^2 \Sec \ge K_1) \ge 1 - K_2 e^{-K_3 \sqrt{d}}.\]
$K_1$, $K_2$, and $K_3$ are positive constants that depend only on $C$.
\end{lem} 

\begin{rem} \label{rem:stiefel}
To sample a 2-dimensional orthonormal frame in $\RR^d$, we can sample from the Stiefel manifold of orthonormal 2-frames as follows: Fill a matrix $A \in \RR^{d\times2}$ with i.i.d.\ normals $\mathcal{N}(0,1)$. Compute the QR factorization of $A$. Then, $Q$ of the QR factorization is a uniformly drawn sample from the Stiefel manifold.
\end{rem}

\begin{proof}[Proof of Lemma~\ref{lem:SecCurvature}]
We recall the expression (\ref{secformula}) for the sectional curvature from the previous section. We easily compute some of the expressions showing up in (\ref{secformula}): \[ \operatorname{Hess} V = \left( \frac{\partial^{2} V}{\partial q_{i} \partial q_{j}} \right) = \Lambda,\qquad \grad V = \left(\frac{\partial V}{\partial q_{1}}, \dots, \frac{\partial V}{\partial q_{n}} \right)^{\tp} = \Lambda q. \]

Note that $2K=\|p\|^2$, where $\| \cdot \|$ is the Euclidean norm. Substituting this expression into the expression~
(\ref{secformula}) gives the full definition of the random variable
\[\Sec =\frac{u\tp \Lambda u}{\|p\|^4} +\frac{v\tp \Lambda v}{\|p\|^4} +\frac{3\|\Lambda q\|^2\cos^2\theta}{\|p\|^6} + \frac{3\|\Lambda q\|^2\cos^2\beta}{\|p\|^6} - \frac{\|\Lambda q\|^2}{\|p\|^6}.\] 

Note that $q$ may be written as $\Lambda^{-1/2} z$, where $z$ is a standard Gaussian vector $z \sim \mcN(0,I_d)$. Therefore
\[ \|\Lambda q\|^2 = z\tp\Lambda z. \]
Next, note that $u$ and $v$ may be written as $u = x/\|x\|$ and $v=y/\|y\|$ where $x$ and $y$ are standard Gaussian vectors (but not independent of each other). 
The terms involving cosines can be left out since they are always positive and small. The other three terms can be written as three quadratic forms in standard Gaussian random vectors $x,y,z$ ($x$ and $y$ are not independent of each other), so that we have
\[ \Sec \ge \frac{x\tp \Lambda x}{\|p\|^4 \|x\|^2} + \frac{y\tp \Lambda y}{\|p\|^4 \|y\|^2} - \frac{z\tp \Lambda z}{\|p\|^6}. \]
We now calculate tail inequalities for all these terms using Chernoff-type bounds.
Let $\lambda_1,\dots,\lambda_d$ be the eigenvalues of $\Lambda$, repeated by multiplicities. By assumption, they are bounded between $1/C$ and $C$, where $C \ge 1$. Let $w_1,\dots,w_d$ be a corresponding set of orthonormal eigenvectors. Let $a_i = x\tp w_i$, so that
\[ x\tp\Lambda x = \sum_{i=1}^d \lambda_i a_i^2. \]
Note that $a_1,\dots,a_d$ are i.i.d.\ $\mcN(0,1)$ random variables. Therefore
\[ \EE(x\tp\Lambda x) = \sum_{i=1}^d \lambda_i = \Trace(\Lambda). \]
Let $S = x\tp \Lambda x - \EE(x\tp\Lambda x)$. Then for any $\theta\in(-1/C,1/C)$,
\[ \EE(e^{\frac{1}{2}\theta S}) = \prod_{i=1}^d \EE(e^{\frac{1}{2}\theta\lambda_i(a_i^2-1)}) = e^{-\frac{1}{2}\theta \sum\lambda_i} \prod_{i=1}^d (1-\theta\lambda_i)^{-1/2}. \]
Now note that if $b\in[0,1/2]$, then $\log(1-b) \ge -b-b^2$. This shows that if we choose $\theta=(2\sum\lambda_i^2)^{-1/2}$, then
\[ \EE(e^{\frac{1}{2}\theta S}) \le e^{\frac{1}{4}}. \]
Thus, for any $t \ge 0$, using Markov's inequality and by assumption $\lambda_i \le C$,
\[ \Prob(S \ge t) \le e^{-\frac{1}{2}\theta t} \EE(e^{\frac{1}{2}\theta S}) \le e^{\frac{1}{4}}e^{-\frac{1}{4}t(\sum\lambda_i^2)^{-1/2}} \le e^{\frac{1}{4}}e^{-\frac{t}{4C\sqrt{d}}}. \]
Similarly the same bound holds for the lower tail, as well as for $y\tp\Lambda y - \EE(y\tp\Lambda y)$ and $z\tp\Lambda z - \EE(z\tp\Lambda z)$, since these random variables have the same distribution as $x\tp\Lambda x - \EE(x\tp\Lambda x)$. By a similar calculation, 
\[ \Prob\left(\big| \|x\|^2 - \EE\|x\|^2 \big| \ge t \right) \le 2e^{\frac{1}{4}}e^{-\frac{t}{4\sqrt{d}}}, \]
and the same bound holds for $\|p\|^2$ and $\|y\|^2$. Note that $\EE\|x\|^2=d$. Let $\mu := \sum\lambda_i/d$. 
Note that $\mu\in[1/C,C]$.
Take some $t$ small enough and let $B$ be the event that all of the following events happen: $|x\tp\Lambda x-\mu d|\le t$, $|y\tp\Lambda y-\mu d|\le t$, $|z\tp\Lambda z-\mu d|\le t$, $|\|x\|^2-d|\le t$, $|\|y\|^2-d|\le t$ and $|\|p\|^2-d|\le t$. By the above calculations, 
\[ \Prob(B) \ge 1-8e^{1/4}e^{-\frac{t}{4C\sqrt{d}}}. \]
If $B$ happens, and $t$ is small enough, then
\[ \Sec \ge \frac{2(\mu d-t)}{(d+t)^3} - \frac{\mu d+t}{(d-t)^3}. \] 
Choose $t=\varepsilon d$ for some sufficiently small constant $\varepsilon$ (depending on $C$), the above inequalities show that
\[ \Prob(d^2\Sec \ge K_1) \ge 1 - K_2 e^{-K_3 \sqrt{d}}, \]
where $K_1$, $K_2$ and $K_3$ are positive constants that depend only on $C$.
\end{proof}

Lemma~\ref{lem:SecCurvature} tells us that sectional curvatures at the \emph{beginnings} of the HMC steps are nearly constant. However, it will also be necessary to show that they tend to remain nearly constant along the HMC steps, at least if we do not move too far on each step. To do this, we show that the values of $\|p\|$ and $\|q\|$ do not change much over the course of an HMC trajectory. 


\begin{lem}
\label{lem:pstable}
Let $C,d,\Sigma,\Lambda$ be as in Lemma~\ref{lem:SecCurvature}. Let $q(0)$ be distributed according to $\pi=\mcN(0,\Sigma)$, and let $p(0)$ be distributed according to $\mcN(0,I_d)$. Let $q(t)$ and $p(t)$ be the trajectory of HMC with initial position $q(0)$ and initial momentum $p(0)$. Then we have \[\PP\left(\sup_{t\in[0,d^{-1/2}]}\Big|\|p(t)\|^2-d\Big|\ge K_4 d^{3/4}\right)\le K_5 e^{-K_6\sqrt{d}}\] for some positive constants $K_4$, $K_5$, and $K_6$, which depend only on $C$.
\end{lem}

\begin{proof}
Since $\|p(0)\|^2$ has a $\chi^2_d$ distribution, by~\cite[Lemma~1, p.~1325]{LM00}, we have \[\PP\left(\Big|\|p(0)\|^2-d\Big|\ge d^{3/4}\right)\le 2e^{-K_7\sqrt{d}}\] for some positive constant $K_7$. Now, choose an orthonormal basis for $\RR^d$ that makes $\Sigma$ diagonal, with diagonal entries $\lambda_1,\ldots,\lambda_d$. With respect to this basis, write $p(t)=(p_1(t),\ldots,p_d(t))$. Solving the Hamiltonian equations~(\ref{hamilton}) gives \[p_i(t)=p_i(0)\lambda_i^{-1/2}\cos(\lambda_i^{-1/2}t)-q_i(0)\lambda_i^{-1/2}\sin(\lambda_i^{-1/2}t).\] Write $a_i$ and $b_i$ for $q_i(0)$ and $p_i(0)$, respectively; the vectors $(a_1,\ldots,a_d)$ and $(b_1,\ldots,b_d)$ are independent. By the tail bounds for the univariate normal distribution, we have \[\PP(|a_i|\le\sqrt{\lambda_i}d^{1/4}, |b_i|\le d^{1/4})\ge 1-4e^{-\sqrt{d}/2}.\] When this happens, \[|p_i'(t)|\le d^{1/4}\lambda^{1/2}\lambda^{-1}+d^{1/4}\lambda^{-1}\le 2Cd^{1/4}.\] Since $\EE(a_i)=\EE(b_i)=0$, we have $\EE(p_i'(t))=0$.

Let $X_i(t)=p_i(t)^2$. We have $X_i'(t)=2p_i(t)p_i'(t)$, so \[|X_i'(t)|=2|p_i(t)|\ |p_i'(t)|\le 2d^{1/4}\times 2Cd^{1/4}=4Cd^{1/2}\] for all $t$, with probability at least $1-6e^{-K_8\sqrt{d}}$ for some constant $K_8$. Furthermore, in this high-probability region, $\EE(X_i'(t))=0$ and $\Var(X_i'(t))\le 16C^2d$.

Let $X(t)=\|p(t)\|^2$. We have $X'(t)=\sum_{i=1}^d X_i'(t)$. Since the $X_i'(t)$'s are independent, we may apply the Lindeberg Central Limit Theorem~\cite{Lindeberg22} so that $X'(t)$ is approximately normal with $\EE(X'(t))=0$ and $\Var(X'(t))\le 16C^2d^2$, when we are in the high-probability region for each $i$, which occurs with probability at least $1-6de^{-K_8\sqrt{d}}$. When this happens, we have \[|X'(t)|<4Cd\times d^{1/4}\] with probability at least $1-2e^{K_9\sqrt{d}}$.

Finally, we have \[X(t)=X(0)+\int_0^t X'(w)\,dw,\] so \[|X(t)-X_0|\le \int_0^t |X'(w)|\, dw\le 4Cd^{5/4}t\le 4Cd^{3/4}\] when we are in the high-probability region. The result follows.
\end{proof}

Thus the distribution of $p$ does not change too much along a single HMC trajectory. Since $q'(t)=p(t)$, the distribution of $q$ also does not change much along a trajectory. Hence the result of Lemma~\ref{lem:SecCurvature} holds if $q$ and $p$ are measured after time $t\in[0,d^{-1/2}]$ rather than just at time 0, the only difference being that the values of the positive constants $K_1$, $K_2$, and $K_3$ are changed.

\begin{rem} 
\label{Rem:MeanCurvature}  
We can easily see from the proof of Lemmas~\ref{lem:SecCurvature} and~\ref{lem:pstable} that the expectation of the sectional curvature is bounded from below by $\EE(\Sec) \ge \frac{\Trace(\Lambda)}{d^3}$.
Furthermore, note that, for instance in Lemma~\ref{lem:SecCurvature}, \[ K_1 = \frac{2(\mu-\varepsilon)}{(\varepsilon+1)^3} - \frac{(\mu+\varepsilon)}{(1-\varepsilon)^3} = \mu - O(\varepsilon), \]
which yields $\Prob( \Sec \ge \frac{\Trace(\Lambda)}{d^3} - O(\varepsilon) ) \ge 1 - K_2 e^{-K_3 \sqrt{d}}$. Therefore, asymptotically, as $d\to\infty$, the sectional curvature is very close to its expected value, as the probability of deviating from it decreases like $K_2e^{-K_3\sqrt{d}}$. 
Since we are interested in truly high dimensional problems, we will work in this asymptotic regime and assume that $\Sec \ge \frac{\Trace(\Lambda)}{(1+\delta)d^3}$ for a suitable $\delta>0$. 
\end{rem}


\begin{table} 
\label{constants} 
\begin{tabular}{lll} 
\textbf{Name} & \textbf{Symbol} & \textbf{Approximate value} \\ 
\hline \\[-0.3cm] 
Coarse Ricci curvature & $\kappa$ & $\displaystyle \frac{\Trace(\Lambda)}{3 d^2}$ \\ 
Coarse diffusion constant & $\sigma(q)^2$ & $d$ \\ Local dimension & $n_q$ & $d$ \\[0.1cm] 
\hline \\ 
\end{tabular} \caption{A table of some expressions that show up in the concentration inequality, together with their approximate values for an identity Gaussian kernel walk on $\RR^d$ with $\pi$ a multivariate Gaussian distribution with precision matrix $\Lambda$.} 
\label{locdimtable} 
\end{table}

Manifolds with constant positive sectional curvature are well-understood: the only such manifolds are spheres and their quotients by discrete groups of isometries. For topological reasons, HMC in high dimensions most closely resembles a Markov chain on a sphere; we will use this intuition for inspiration.


In high dimensions, the Gaussian $\mcN(0,I_d)$, the distribution of momenta, is very close to the uniform distribution on the sphere of radius $\sqrt{d}$, and by the previous results, a trajectory in the Jacobi metric is very close to being a geodesic on the sphere. Thus, we now compute the Wasserstein distance between spherical kernels centered at two nearby points on $\SS^d$. 
Following that, the next step is to estimate Wasserstein distances of Gaussian kernels on a sphere, 
and then analyze how these Wasserstein distances deviate from Wasserstein distances of Gaussian kernels on HMC manifolds.

We make use of the fact that, if $x$ and $y$ are two points on $\SS^d$, then the spherical distance between $x$ and $y$ is $\cos^{-1}(x\cdot y)$.


\begin{lem}
\label{lem:Wasserstein-sphere}
Let $\SS^d$ denote the unit $d$-dimensional sphere. Let $x$ and $y$ be two points on $\SS^d$ of distance $\eps>0$. Let $P_x(r)$ and $P_y(r)$ denote the uniform measure on the set of points of $\SS^d$ of spherical distance $r$ from $x$ and $y$, respectively. Then \[W_1(P_x(r),P_y(r))\le\eps\left(1-\frac{\sin^2(r)}{2}\frac{d-1}{d}+O(\sin^4(r))\right).\]
\end{lem}

\begin{proof}
By rotating the sphere if necessary, we may assume that $x=(1,0,\ldots,0)$ and $y=(\cos\eps,\sin\eps,0,\ldots,0)$. Next, we pick a point $x'$ of distance $r$ from $x$; a typical such point has the form \[x'=(\cos(r),\sin(r)\cos(\phi_1),\sin(r)\sin(\phi_1)\cos(\phi_2),\ldots,\sin(r)\sin(\phi_1)\cdots\sin(\phi_{d-1})).\] In order to obtain an upper bound for the Wasserstein distance between the spherical kernels centered at $x$ and $y$, we must devise a transference plan, sending $x'$ to some $y'$ in the spherical kernel centered at $y$. We do not know the optimal transference plan, but one good method sends $x'$ to $y'$, where \[y'=\begin{pmatrix}\cos\eps & -\sin\eps & 0 & \cdots & 0 \\ \sin\eps & \cos\eps & 0 & \cdots & 0 \\ 0 & 0 & 1 & \cdots & 0 \\ \vdots & \vdots & 0 & \ddots & 0 \\ 0 & 0 & 0 & \cdots & 1 \end{pmatrix} (x'){\tp}.\] In short, rotate the first two coordinates by $\eps$, and leave the other coordinates unchanged.


Computing using the spherical distance formula is challenging, so instead, we imagine moving from $x'$ to $y'$ along a parallel, i.e.\ the path \[\gamma(t)=\begin{pmatrix}\cos t& -\sin t & 0 & \cdots & 0 \\ \sin t & \cos t & 0 & \cdots & 0 \\ 0 & 0 & 1 & \cdots & 0 \\ \vdots & \vdots & 0 & \cdots & 0 \\ 0 & 0 & 0 & \cdots & 1 \end{pmatrix} (x')\tp,\qquad 0\le t\le\eps.\] This is a longer path than is obtained by traveling along a great circle, so it only gives us an upper bound on the Wasserstein distance. The length of $\gamma(t)$ is \[L(\gamma)=\eps\sqrt{1-\sin^2(r)\sin^2(\phi_1)}=\eps\left(1-\frac{\sin^2(r)\sin^2(\phi_1)}{2}+O(\sin^4(r))\right).\] The Wasserstein distance between $x$ and $y$ with spherical kernels of radius $r$ is then at most \[\frac{\int_0^\pi \eps\left(1-\frac{\sin^2(r)\sin^2(\phi_1)}{2}+O(\sin^4(r)\right)\sin^{d-2}(\phi_1)\, d\phi_1}{\int_0^\pi \sin^{d-2}(\phi_1)\, d\phi_1} = \eps\left(1-\frac{\sin^2(r)}{2}\frac{d-1}{d}+O(\sin^4(r))\right).\]
\end{proof}

From here, we can compute the coarse Ricci curvature:

\begin{cor} \label{cor:gaussric} The coarse Ricci curvature $\kappa(x,y)$ with respect to the kernels $P_x(r)$ and $P_r(y)$ is \[\kappa(x,y)=1-\frac{\rho(P_x(r),P_y(r))}{\eps}\ge \frac{\sin^2(r)}{2}\frac{d-1}{d}+O(\sin^4(r)).\]
\end{cor}

In Lemma~\ref{lem:Wasserstein-sphere} and Corollary~\ref{cor:gaussric}, we compute the Wasserstein distance and coarse Ricci curvature with respect to a spherical kernel. Note that the Wasserstein distance increases, and the coarse Ricci curvature decreases, as $r$ decreases. Ultimately, we need to compute Wasserstein distance with respect to a Gaussian kernel. By tail bounds on the $\chi^2$ distribution given in~\cite{LM00}, we have, for instance, \[\PP\left(\Large|\|p\|^2-d\Large|\ge 4d^{3/4}\right)\le 2e^{-d^{1/2}}.\] If $d\gg 0$, then with very high probability, we do not encounter any $p$ with $\Large|\|p\|-d\Large|>4d^{3/4}$, so these momenta do not impact the running of the chain. From now on, we assume that any $p$ we encounter has $\Large|\|p\|-d\Large|\le 4d^{3/4}$.

Working in this asymptotic regime, we assume that the HMC for the Gaussian is a random walk on a $d$-dimensional sphere with sectional curvature $\frac{\Trace(\Lambda)}{d^3}$ and hence radius $\frac{d^{3/2}}{\Trace(\Lambda)^{1/2}}$. A typical trajectory of HMC moves at speed $\|p(0)\|\approx\sqrt{d}$ \emph{in Euclidean space} for time $\frac{1}{\sqrt{d}}$ and hence moves a distance of roughly 1 in Euclidean space. The Jacobi metric increases distances by a factor of $\sqrt{2(h-V)}=\sqrt{K}=(1+O(d^{-1/4}))\sqrt{d}$, so an HMC trajectory looks like a geodesic on the sphere of radius $\frac{d^{3/2}}{\Trace(\Lambda)^{1/2}}$ of length $(1+O(d^{-1/4}))\sqrt{d}$. After scaling the sphere to have unit radius, the geodesic then has length $(1+O(d^{-1/4}))\frac{\Trace(\Lambda)^{1/2}}{d}$. Thus, in the above analysis, we take $r$ to be $(1+O(d^{-1/4}))\frac{\Trace(\Lambda)^{1/2}}{d}=O(d^{-1/2})$ and use the approximation $\sin(r)\approx r$. Absorbing the $O(d^{-1/4})$ into the denominator, we have shown the following:

\begin{cor} \label{cor:crcspheregauss} The coarse Ricci curvature of a random walk on $\SS^d$ with a Gaussian kernel with covariance matrix $\frac{\Trace(\Lambda)}{d^{1/2}}I_d$ is bounded below by \[\kappa(x,y)\ge\frac{3\Trace(\Lambda)}{8d^2}+O(d^{-2}).\] \end{cor}


Now, we wish to compare the Wasserstein distance of Gaussian kernels on a sphere to the Wasserstein distance of Gaussian kernels on HMC manifold with the Jacobi metric. 

Pick two nearby point $x,y\in\RR^d$, and let $h$ be such that $|h-V(x)-d|<d^{3/4}$. Write $\mcX$ for $\RR^d$ equipped with the Jacobi metric $g_h$. Let $\gamma:[0,\eps]\to\mcX$ be a geodesic with $\gamma(0)=x$ and $\gamma(\eps)=y$. Let $u\in T_x\mcX$ be a tangent vector, short enough so that there are no conjugate points on the geodesic $\exp_x(tu)$ for $t\in[-2,2]$. For $t\in[0,\eps]$, let $u(t)$ be the parallel transport of $u$ along $\gamma$, so that $u(t)\in T_{\gamma(t)}\mcX$. Let $\omega_s(t)=\exp_{\gamma(t)}(su(t))$. (See Figure~\ref{ptgamma} for a diagram.) Now, we make a similar construction on the $d$-dimensional sphere $S$ with radius $\frac{9\Trace(\Lambda)}{8d^3}$: let $\widetilde{x},\widetilde{y}\in S$ be two points of geodesic distance $\eps$ along a geodesic $\widetilde{\gamma}$. Let $\widetilde{u}\in T_{\widetilde{x}}S$ be a tangent vector with $\|\widetilde{u}\|=\|u\|$ and $\langle \widetilde{u},\widetilde{\gamma}'(0)\rangle=\langle u,\gamma'(0)\rangle$. Let $\widetilde{u}(t)$ be the parallel transport of $\widetilde{u}$ along $\widetilde{\gamma}$, so that $\widetilde{u}(t)\in T_{\widetilde{\gamma}(t)}S$, and let $\widetilde{\omega}_s(t)=\exp_{\widetilde{\gamma}(t)}(s\widetilde{u}(t))$. We wish to compare the lengths $L_s$ and $\widetilde{L}_s$ of $\omega_s(t)$ and $\widetilde{\omega}_s(t)$, where $t$ ranges over $[0,\eps]$. Observe that the sectional curvatures of $\mcX$ are greater than those of $S$, so we are in the setting of classical comparison theorems in differential geometry, in particular the Rauch comparison theorem, which we use in the proof of the theorem below.

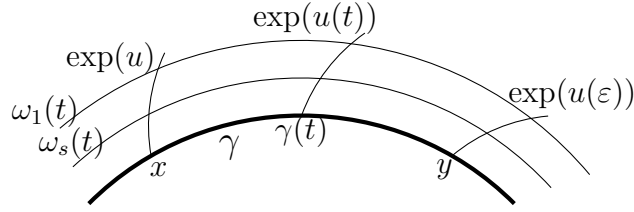
\begin{figure}
\begin{tikzpicture}
\draw[ultra thick] (45:4) arc (45:135:4);
\path node at (120:3.8) {$x$};
\path node at (90:3.8) {$\gamma(t)$};
\path node at (60:3.8) {$y$};
\path node at (105:3.7) {\large $\gamma$};
\draw (120:4) arc (190:155:2.3);
\path node at (118:5.4) {$\exp(u)$};
\path node at (88:5.3) {$\exp(u(t))$};
\path node at (50:5.6) {$\exp(u(\eps))$};
\draw (90:4) arc (160:125:2.3);
\draw (60:4) arc (130:95:2.3);
\draw (42.5:4.5) arc (42.5:132.5:4.5);
\draw (40:5) arc (40:130:5);
\path node at (130:4.75) {$\omega_s(t)$};
\path node at (130:5.3) {$\omega_1(t)$};
\end{tikzpicture}
\caption{Parallel transport of $u$ along $\gamma$}
\label{ptgamma}
\end{figure}

\begin{prop} For each $s\in[0,1]$, $L_s\le \widetilde{L}_s$. \end{prop}

\begin{proof} In the case that $\langle u,\gamma'(0)\rangle=0$, this is a special case (with $f(t)=\|u\|$ a constant function) of~\cite[Corollary 1.36]{CE08}, a consequence of the Rauch comparison theorem. When this is not the case, we modify the proof given there. For $t\in [0,\eps]$ and $s\in [0,1]$, let $\eta_t(s)=\exp_{\gamma(t)}(su(t))$; define $\widetilde{\eta}_t(s)$ similarly on $S$. The family of tangent vectors $U=\omega'_s(t)$ and $\widetilde{U}=\widetilde{\omega}'_s(t)$ are Jacobi fields. Now, a Jacobi field $U$ along $\eta_s$ can be decomposed as a sum $U^{\|}+U^\perp$ of vector fields, where $U^{\|}$ is parallel to $\eta$ and $U^\perp$ is orthogonal, as shown in Figure~\ref{fig:jacdec}. Similarly, $\widetilde{U}$ can be decomposed as a sum $\widetilde{U}^\|+\widetilde{U}^\perp$.

It is known (see for instance~\cite[Corollary B.14, pp.\ 292--293]{CK04}) that any Jacobi field $J(t)$ along a geodesic $\eta(t)$ has a unique orthogonal decomposition $J(t)=J^\perp(t)+(at+b)\eta'(t)$, for some $a,b\in\RR$, where $\langle J^\perp(t),\eta'(t)\rangle=0$. Thus, in our case, for each $s$, $U_s^\|(t)=(a_st+b_s)\eta'_s(t)$. Similarly, $\widetilde{U}_s^\|(t)=(\widetilde{a}_st+\widetilde{b}_s)\widetilde{\eta}'_s(t)$. By construction, we have $a_s=\widetilde{a}_s$ and $b_s=\widetilde{b}_s$. Hence $\|U^\|\|=\|\widetilde{U}^\|\|$. By the Rauch comparison theorem~\cite[Thereom 1.34]{CE08}, $\|\widetilde{U}^\perp(\widetilde{\eta}_t(1))\|\ge\|U^\perp(\eta_t(1))\|$, for each $t$. Combining the parallel and orthogonal components, we have $\|\widetilde{U}(\widetilde{\eta}_t(1))\|\ge\|U(\eta_t(1))\|$. Since $L_s=\int_0^\eps \|U(\omega_s(t))\|\, dt$ and similarly for $\widetilde{L}_s$, the proposition follows. \end{proof}

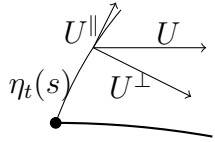
\begin{figure}
\begin{tikzpicture}
\filldraw (0,0) circle (2pt);
\draw[thick] (0,0) arc (90:80:12);
\draw (0,0) arc (160:140:5);
\path node at (-.2,.5) {$\eta_t(s)$};
\draw[->] (.5,1) -- (2,1);
\path node at (1.5,1.2) {$U$};
\draw[->] (.5,1) -- (1.8,.35);
\path node at (1,.5) {$U^\perp$};
\draw[->] (.5,1) -- (.8,1.6);
\path node at (.35,1.25) {$U^\|$};
\end{tikzpicture}
\caption{Decomposition of the Jacobi field $V=W+W^\perp$. $W$ is parallel to $\eta$.}
\label{fig:jacdec}
\end{figure}

As a consequence, we can compare transference plans of Gaussian kernels on $S$ and $\mcX$: as long as we avoid the exceptional sets, the distance between points in neighborhoods of $x$ and $y$ in $\mcX$ are less than those of the corresponding points on $S$. After scaling, the transference plan on $S$ described in the proof of Lemma~\ref{lem:Wasserstein-sphere} is exactly the path taken by $w_0$. Thus, the Wasserstein distance on $\mcX$ is lower than that of $S$, so the coarse Ricci curvature is higher. Comparing with Corollary~\ref{cor:crcspheregauss} and using the scaling described below Corollary~\ref{cor:crcspheregauss}, we find:

\begin{thm} \label{thm:crc-hmc-gaussian} The coarse Ricci curvature of $\mcX$ satisfies \[\kappa\ge\frac{\Trace(\Lambda)}{3d^2}+O(d^{-2}).\] \end{thm}

\begin{prop} The coarse diffusion constant $\sigma(q)^2$ is $d$. \end{prop}

\begin{proof} Since the transition kernel is insensitive to the starting point, we may assume that $q=0$. We have \begin{align*} \sigma(q)^2 &= \frac{1}{2} \iint_{\RR^d\times\RR^d} \frac{1}{(2\pi)^d} \exp\left(-\frac{1}{2}(\|x\|^2+\|y\|^2)\right) \|x-y\|^2\, dV(x)\, dV(y) \\ &= \frac{1}{2(2\pi)^d} \iint_{\RR^d\times\RR^d} \exp\left(-\frac{1}{2}(\|x\|^2+\|y\|^2)\right)\sum_i (x_i^2-2x_iy_i+y_i^2)\, dV(x)\, dV(y) \\ &= \frac{1}{2(2\pi)^d}\int_{\RR^{2d}} e^{-\|z\|^2/2}\|z\|^2\, dV \\ &= \frac{SA_{2d-1}}{2(2\pi)^d}\int_0^\infty e^{-r^2/2}r^2 r^{2d-1}\, dr \\ &= d. \end{align*} \end{proof}

Based on \cite{Ollivier09}, it follows that the local dimension $n_q$ is $d+O(1)$.


We show now how the coarse Ricci curvature, coarse diffusion constant, local dimension, and eccentricity from Table~\ref{constants} can be used to calculate concentration inequalities for a specific example. We focus on the Gaussian distribution with weak dependencies between variables. But first we need to introduce two propositions.

\begin{prop} For any $\eps>0$, for $d\gg 0$, the granularity $\sigma_{\infty}$ of HMC with Gaussian target distribution is $(1+\eps) \sqrt{d}$. \end{prop}

Recall that, as usual, we interpret $\sigma_\infty$ in such a way as to cut out points in the support of $P_x$ of exceptionally low probability. This computation was already done in the paragraph preceding Corollary~\ref{cor:crcspheregauss}.

\begin{prop} \label{prop:Lischitz}
For $d\gg 0$, the Lipschitz norm of a coordinate function 
\[ f_i: q \to q_i \]
of a HMC random walk with Gaussian measure $\pi$ on $\RR^d$ as in Theorem \ref{thm:ConcentrationInequality} is bounded above by $\frac{2}{\sqrt{d}}$.
\end{prop}

\begin{proof}
Starting with Definition \ref{lipschitz}
\[ \|f\|_{\Lip} := \sup_{x,y \in\X } \frac{|f(x)-f(y)|}{\rho(x,y)} = \sup\frac{1}{\sqrt{2K}} = \|p\|^{-1}. \]
The inverse momentum $\|p\|^{-1}$ follows an inverse $\chi$ distribution with $d$ degrees of freedom. Its expectation and variance are given by 
\[ \EE(\|p\|^{-1}) = \frac{\Gamma((d-1)/2)}{\sqrt{2}~\Gamma(d/2)}, \qquad
\Var(\|p\|^{-1}) = \frac{1}{d-2} - \EE(\|p\|^{-1})^2. \]
The variance is small in high dimensions, so away from an exceptional set, $\|p\|^{-1}$ is very close to $\EE(\|p\|^{-1})$. Thus for $d\gg 0$, we have $\|f\|_{\Lip}\le\frac{1+\eps}{\sqrt{d}}$ away from the exceptional set.
\end{proof}

Everything so far assumes that we stay away from the exceptional set at all times. Let us now tabulate all the places we had to assume that we were avoiding an exceptional set, and the probability of landing in the exceptional set:
\begin{enumerate}
\item Computation of the Lipschitz norm $\|f\|_{\Lip}$
\item Computation of the granularity $\sigma_\infty$
\item Lemma~\ref{lem:SecCurvature}
\item Lemma~\ref{lem:pstable}
\item Comparing HMC to the Gaussian walk on the sphere, following Corollary~\ref{cor:crcspheregauss}.
\end{enumerate}
In each case, the exceptional set occurs with probability $O(e^{-cd^{1/2}})$ for some $c>0$. Assuming we take $T$ steps, the probability we ever hit the bad region is $O(Te^{-cd^{1/2}})$. It is necessary to add this exceptional probability to the bound we get from Theorem~\ref{thm:ConcentrationInequality}. We obtain the following:

\begin{cor} \label{cor:probforgaussian} Under the hypotheses of Theorem~\ref{thm:ConcentrationInequality}, we have \[\PP_x(|\widehat{I}-\EE_x\widehat{I}|\ge r\|f\|_{\Lip})\le 2e^{-r^2/(16V^2(\kappa,T))}+O\left(Te^{-cd^{1/2}}\right)\] for $r<\frac{4V^2(\kappa,T)\kappa T}{3\sigma_\infty}$. \end{cor}

Plugging in the bounds obtained, assuming that $T_0=0$, and taking $d$ large enough so that our bounds absorb some of the error terms, we obtain \[\PP_x(|\widehat{I}-\EE_x\widehat{I}|\ge r\|f\|_{\Lip})\le 2e^{-Tr^2\Trace(\Lambda)^2/(144d^4)}+O\left(Te^{-cd^{1/2}}\right).\] Since $\Trace(\Lambda)=\Theta(d)$, for a fixed error bound $r$, it is necessary to take $T$ scaling as $d^2$ to make the first term small. For large $d$, this choice of $T$ will also keep the big $O$ term small. For a family of functions $f$ whose (Euclidean) Lipschitz norm is constant as $d\to\infty$, we can improve matters by keeping $r\|f\|_{\Lip}$ constant; since $\|f\|_{\Lip}=\Omega(d^{-1/2})$, we can take $T=Cd$, albeit with a large constant $C$. In general, the bounds give good scaling with dimensionality, but poor constants; there is surely much room to improve the constants.

Now we are ready to go through a simple example in 100 dimensions. In 100 dimensions, the bound we give on the probability of landing in the exceptional set is too high, but we ignore this issue and focus on the main term for now; in higher dimensions, the probability of landing in the exceptional set drops off subexponentially in the dimension.

Our aim is to sample from an $100$-dimensional multivariate Gaussian $\mcN( 0, \Sigma )$ with covariance matrix $\Sigma_{ij} = \exp(-|i-j|^2)$, so as to obtain an error of at most $0.05$ with high probability. We use the following HMC parameters, constants taken from Table \ref{constants} and from the computations in this section: 
\begin{center}
\vspace{0.2cm}
\begin{tabular}{ll}
\hline \\[-0.3cm]
Error bound & $r = 0.25$ \\ 
Starting point & $q_0 = 0$ \\
Markov chain kernel & $P \sim \mcN(0,I_{100})$ \\
Coarse Ricci curvature & $\kappa = 0.0048$ \\ 
Coarse diffusion constant & $\sigma^2(q) = 100$ \\
Granularity & $\sigma_\infty=20$ \\
Local dimension & $n_q = 100$ \\
Lipschitz norm & $\|f\|_{\Lip} = 0.2$ \\ 
\hline
\end{tabular}
\vspace{0.2cm}
\end{center}
In our example, the observable function $f$ is the first coordinate function
\[
I = \int_{\RR^{100}} q_1\,\pi(dq),
\]
so the correct solution to this integral is $I=0$. We start the HMC chain at the center of the distribution. 
In Figure \ref{fig:WeakDependenceCovariance} on the right, we see our theoretical concentration inequality as a function of the running time $T$ (in logarithmic scale).
The probability of making an error above our defined error bound $r=0.05$ is close to zero at running time $T=10^8$. 
\begin{figure}
    \centering
    \includegraphics[width=0.4\textwidth]{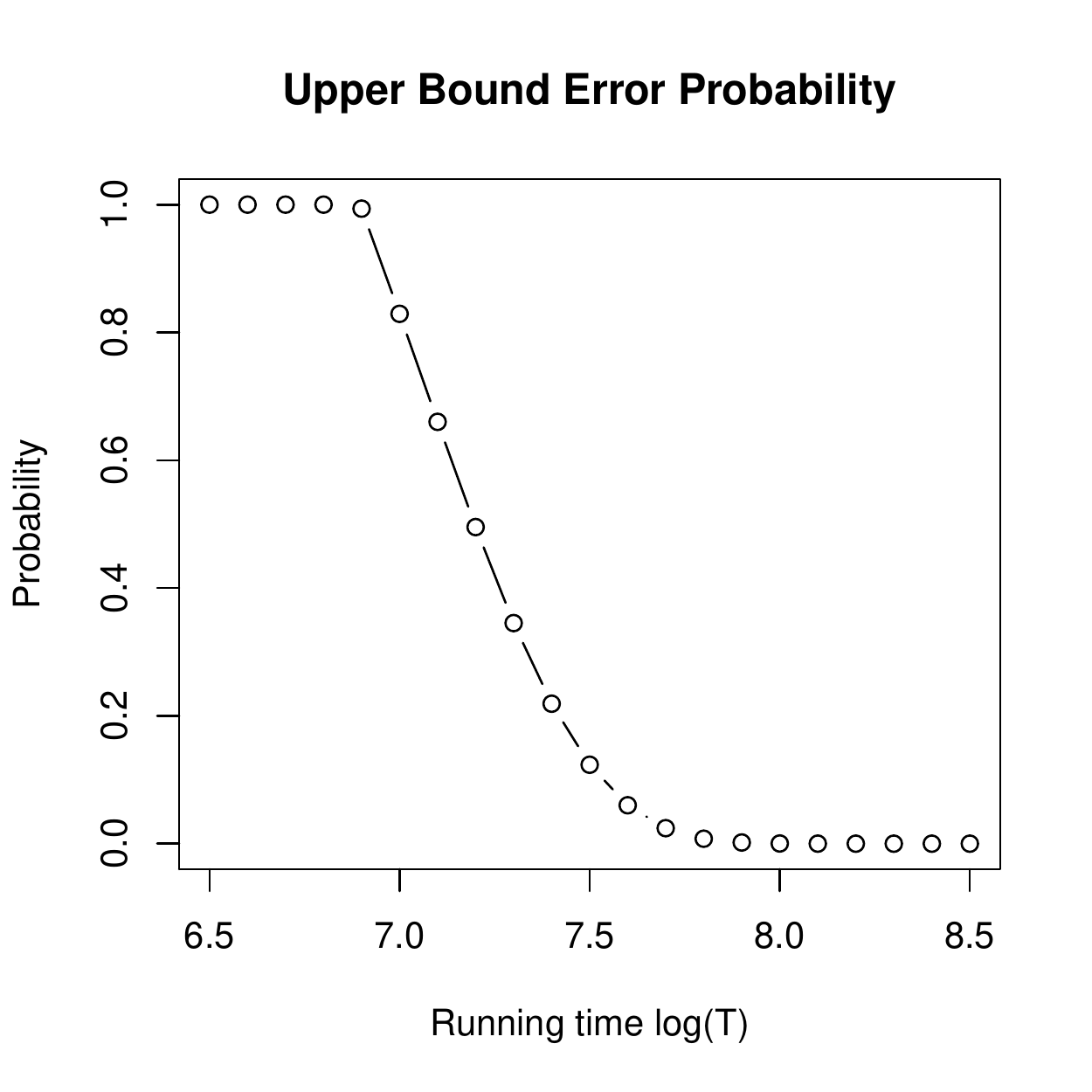}
    \caption{(Covariance structure with weak dependencies) 
    Concentration inequality. 
    }
    \label{fig:WeakDependenceCovariance}
\end{figure}

\subsection{Multivariate $t$ Distribution}
\label{sec:tDistribution}

In the previous section, we showed how to obtain concentration results for HMC Markov chains by using asymptotic sectional curvature estimates as $d\to\infty$. This was possible since our target distribution was a multivariate Gaussian, for which we proved that the probability of deviating from $\EE(\Sec)$ decreases like $K_2e^{-K_3\sqrt{d}}$ as $d\to\infty$. For most distributions of interest, e.g.\ posterior distributions in Bayesian statistics, such a proof is not feasible. In these cases, we propose to compute the empirical sectional curvature distribution and use the sample mean or sample infimum as a numerical approximation. Besides the practical benefits, as mentioned in Remark~\ref{riccicheating}, computing empirical curvatures ignores unlikely curvatures that we never see in practice. 

To illustrate this, we show how it can be done for the multivariate $t$ distribution
\[
\pi(q) = \frac{ \Gamma((\nu+d)/2) }{ \Gamma(\nu/2) \sqrt{\det(\Sigma) (\nu \pi)^d} } \left(1+ \frac{ q\tp \Sigma^{-1} q }{\nu}\right)^{-(\nu+d)/2}.
\]
Here $d$ is the dimension of the space, and $\nu$ is the degrees of freedom. Let $\Sigma$ be the covariance matrix. Let us write $\Sigma^{-1} = (a_{ij})$. Write $Q(q)$ for the quadratic form $q\tp \Sigma^{-1} q = \sum_{i,j} a_{ij} q_iq_j$. Hence, we can take the potential energy function $V$ to be 
\[
V(q) = \frac{\nu+d}{2} \log\left(1 + \frac{Q(q)}{\nu}\right).
\]
The gradient is a $d$-dimensional vector, whose $i^\text{th}$ component is
\[
\frac{\partial V}{\partial q_i} = (\nu+d) \frac{ \Sigma^{-1} q}{Q(q)+\nu}.
\]
The Hessian is a $d$ by $d$ matrix whose $ij$ component is
\[
\frac{\partial^2 V}{\partial q_i \partial q_j} = (\nu+d) \frac{ a_{ij}(Q(q)+\nu) - 2 (\sum_\ell (a_{i\ell} q_\ell))(\sum_m (a_{jm}q_m)) }{ (Q(q)+\nu)^2 }.
\]

Figure \ref{fig:EmpiricalCurvature_tDistribution} shows empirical sectional curvature distribution for different values of $\nu$ in dimension $d=100$.
\begin{figure}
    \centering
    \includegraphics[width=0.32\textwidth]{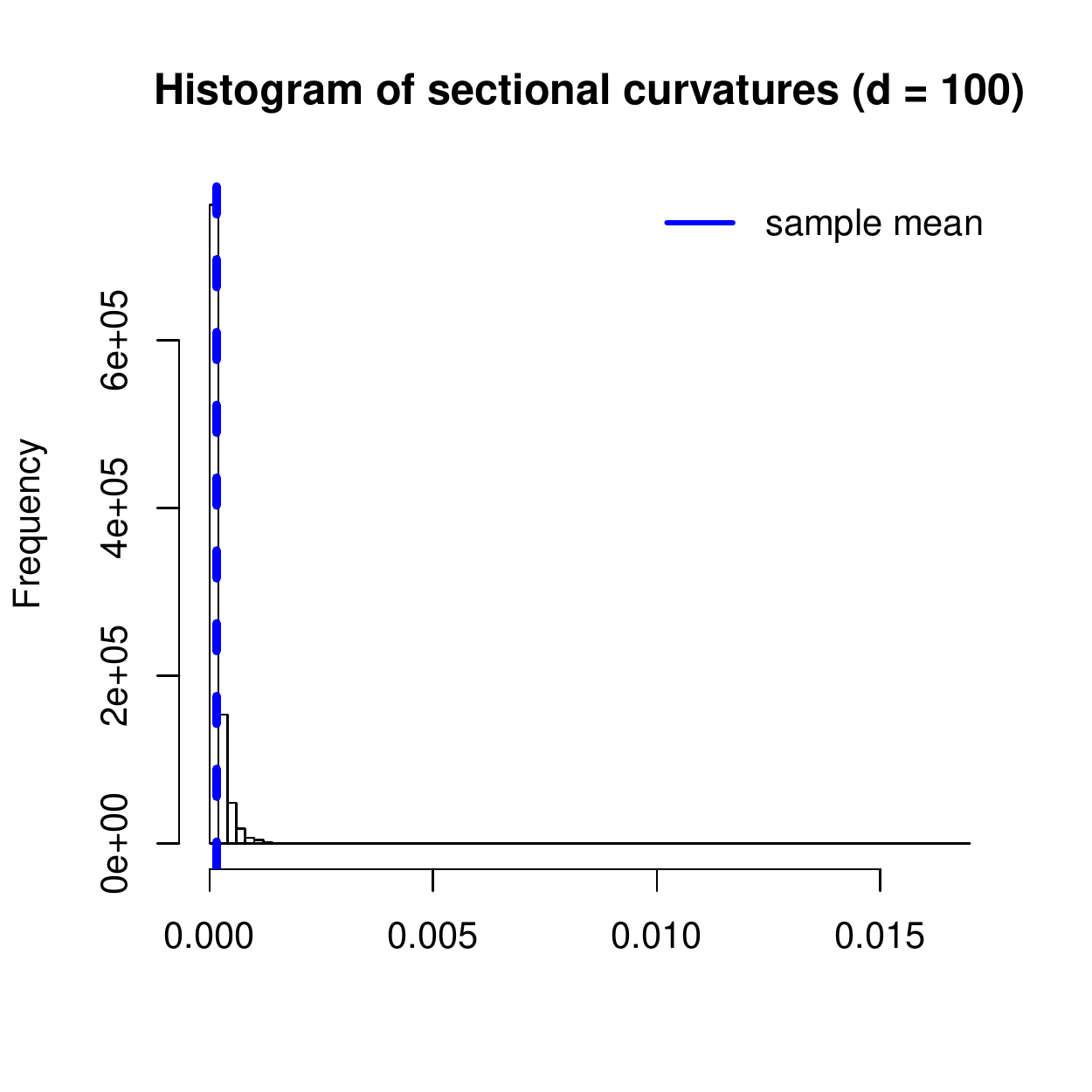}
    \includegraphics[width=0.32\textwidth]{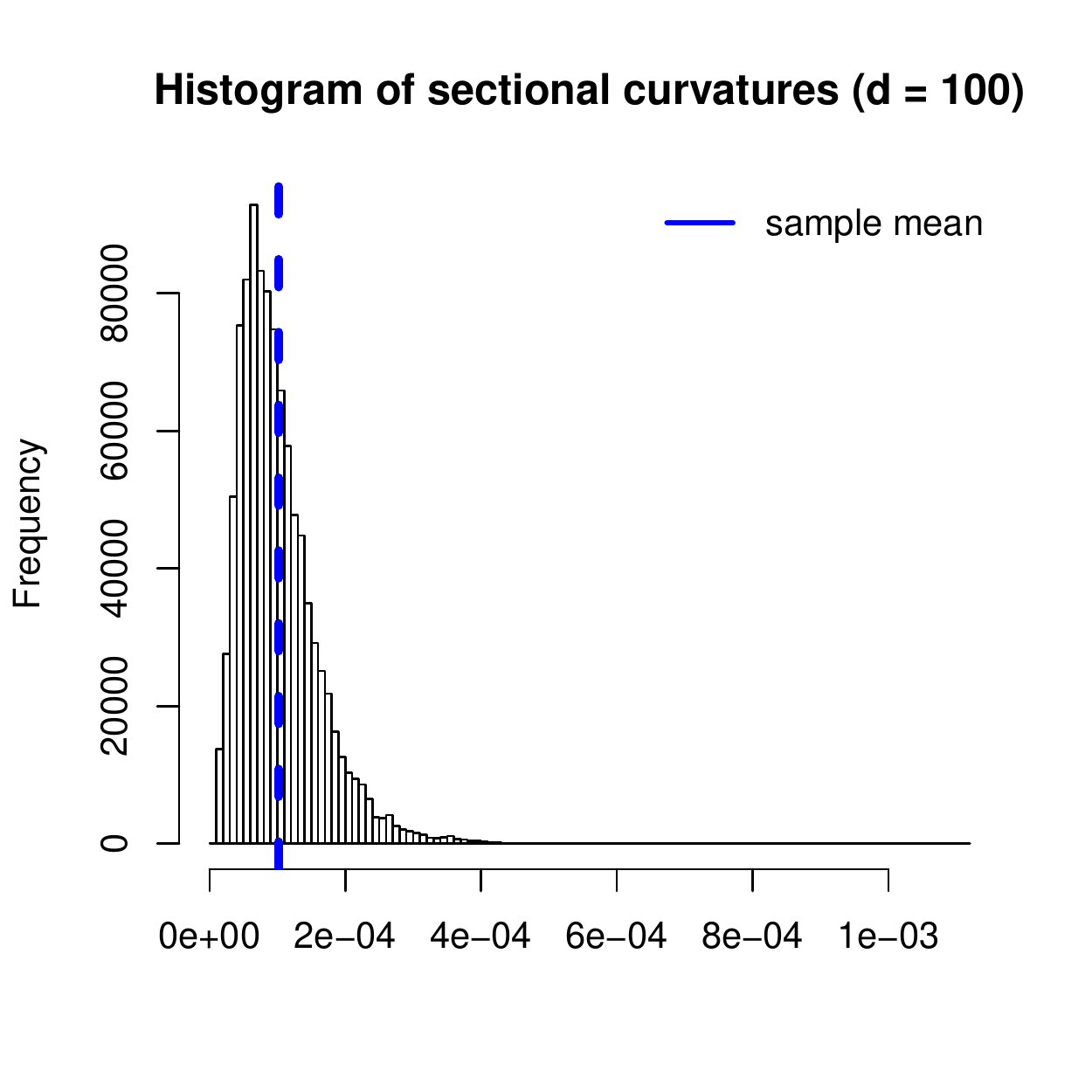}
    \includegraphics[width=0.32\textwidth]{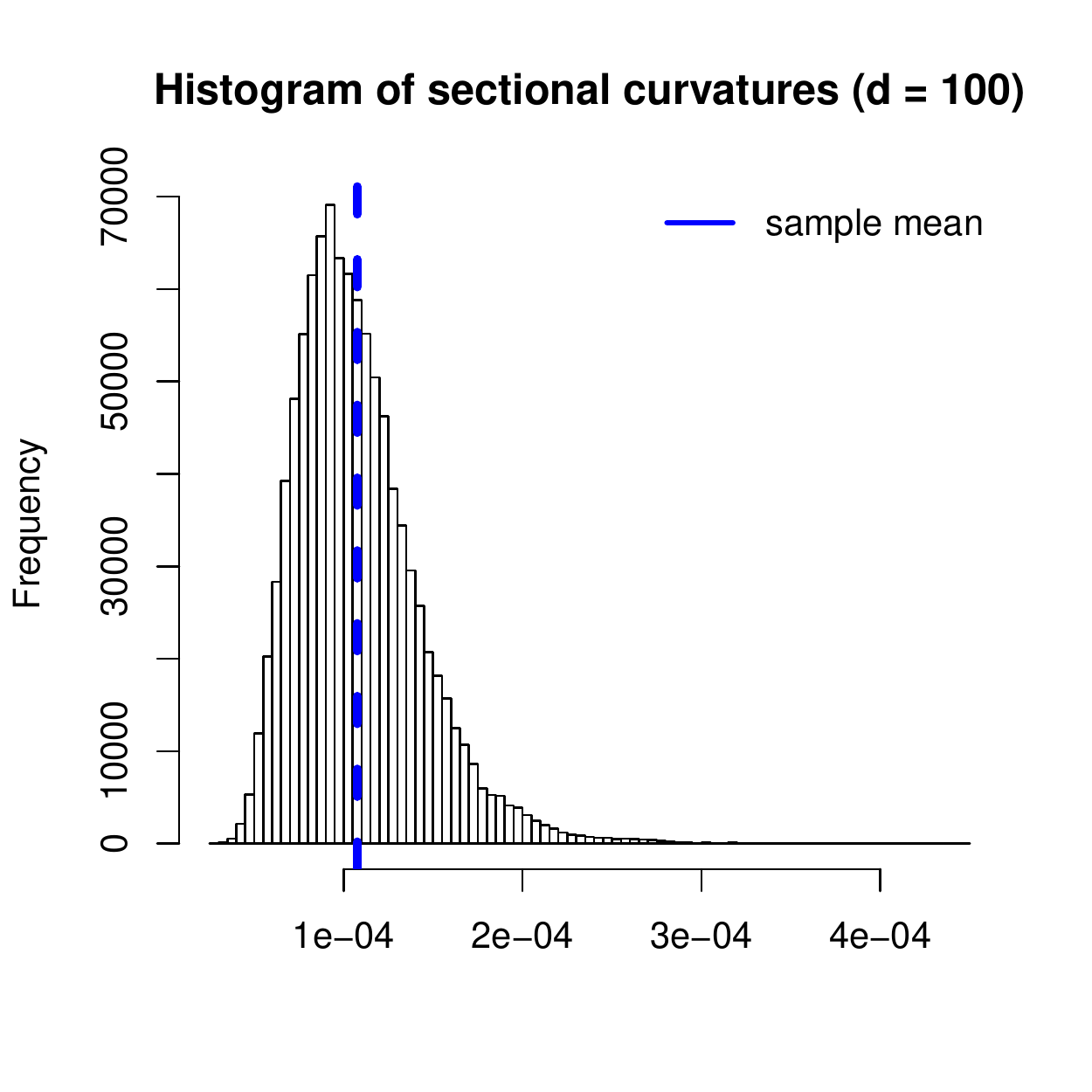}
    \caption{Empirical sectional curvature distribution for multivariate $t$ distribution with $\nu=1$ (top-left), $\nu=10$ (top-right), $\nu=100$ (bottom), $T_0=0$, $T=10^4$ and $d=100$.}
    \label{fig:EmpiricalCurvature_tDistribution}
\end{figure}
With increasing degrees of freedom $\nu$, we approach the curvature distribution of the Gaussian example; compare with Figure \ref{fig:IdentityCovarianceDimensions}. This makes sense, as we know that as $\nu \to \infty$, the multivariate $t$ distribution converges to a multivariate Gaussian. For lower degrees of freedom, the curvature is more spread out. This can be explained by the larger tails of the $t$ distribution and the sharper peak at its mode. Intuitively, larger tails means more equidensity regions of the space, and so we would expect curvature values closer to zero when we are far from the mode. Similarly, we get higher curvature around the mode of the distribution.

To transfer the sample curvature results into concentration inequalities, we can now pick either the sample mean or the sample infimum of the curvature sample distribution. For instance, in the case of $\nu=100$, we would take the sectional curvature to be around $10^{-4}$, which would give us similar concentration inequality figures as in the previous section for the Gaussian case.

\subsection{Bayesian Image Registration} \label{sec:anatomy}

In the previous section, we were able to compute empirical sectional curvatures from known closed-form gradients and Hessians of the potential function $V$. In applied problems, the analytical form of the Hessian is usually hard to derive and even harder to compute. A common way to approximate it is by first order Taylor expansion \cite{GM78}; in numerical methods literature this is referred to as the Gauss-Newton approximation to the Hessian. We will come back to what kind of conditions are needed for this approximation to make sense in our application; see \cite{AJS07a} for more details.

In this section, we use Gauss-Newton approximation to compute the Hessian of a real-world medical image problem and to compute empirical curvature similarly to what we did for the multivariate $t$ distribution. This allows us to obtain concentration inequalities for a Bayesian approach to medical image registration. 

The goal in medical image registration is to find a deformation field that maps one image to another image. For instance, these deformations are then used for the statistical analysis of shape differences between groups of patients. After introducing a basic mathematical formulation of the registration problem, we show how our concentration results can be used as a diagnostic tool for HMC complementary to other tools like visual assessment via trace plots or statistical tests of Markov chains \cite{GR92}.

\begin{figure}
    \centering
    \includegraphics[width=0.2\textwidth]{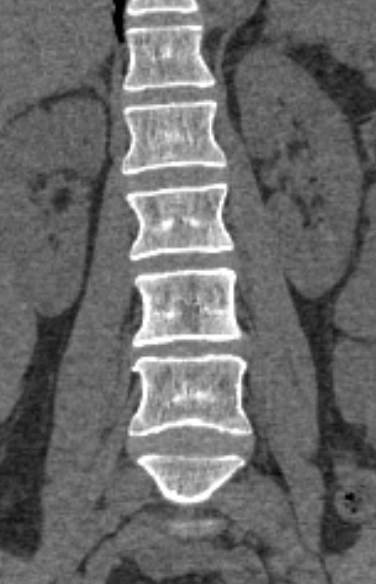}
    \includegraphics[width=0.2\textwidth]{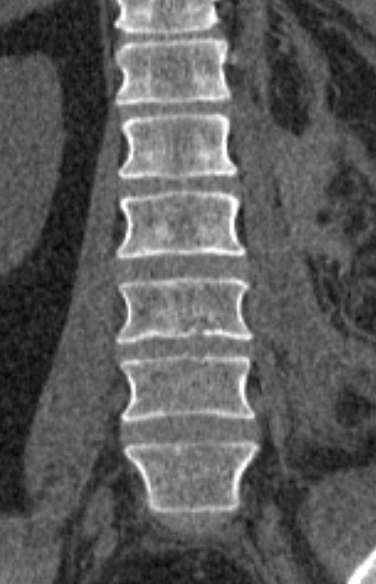}
    \caption{Two-dimensional slices extracted from three-dimensional volumetric computed tomography (CT) images of two patients. Frontal plane view with all five lumbar vertebrae and one thoracic vertebra visible.}
    \label{fig:FixedMovingImage}
\end{figure}
We explain the problem of medical image registration with a two-dimensional example. We extracted two frontal slices from computed tomography (CT) images; see Figure \ref{fig:FixedMovingImage}. The goal is to spatially deform the moving image (Figure \ref{fig:FixedMovingImage}, right) to the fixed image (Figure \ref{fig:FixedMovingImage}, left). The voxel coordinates $(x,y)$ are related by a global affine transformation $A$ and a local deformation field $\varphi(x,y,q)$ ($q$ are the parameters of the deformation) of the form
\[
\begin{bmatrix} x'(q) \\ y'(q) \\ 1 \end{bmatrix} = A \begin{bmatrix} x \\ y \\ 1 \end{bmatrix} + \begin{bmatrix} \varphi_x(x,y,q) \\ \varphi_y(x,y,q) \\ 1 \end{bmatrix}.
\]
Here we will assume that $A$ is known and that we are only estimating the local deformation $\varphi(x,y,q)$.
We choose to parametrize the deformations $\varphi(x,y,q)$ using cubic B-splines with coefficients $q$ and follow the presentation by Andersson, Jenkinson and Smith \cite{AJS07}.  
Let $(q_{i,j})^{(x)}$ denote the spline weights in direction $x$ at control points $(i,j)$, and similarly for $y$.
Then we reshape the two matrices into a column vector 
\[ q = \begin{bmatrix} \operatorname{Vec}\left((q_{i,j})^{(x)}\right) \\ \operatorname{Vec}\left((q_{i,j})^{(y)}\right) \end{bmatrix}, \]
where $\operatorname{Vec}$ takes each row of the matrix and concatenates it into a column vector.
We can write any deformation as a linear combination of tensor products of one dimensional cubic B-splines~\cite{Rueckert99}:
\begin{equation} \label{equ:SplineDeformation}
\varphi(x,y,q) = \sum_{\ell=0}^3 \sum_{m=0}^3 B_\ell(u) B_m(v) q_{i+\ell,j+m},
\end{equation}
where the sum goes over $16$ neighboring control points with indices calculated as 
$i=\lfloor x/n_x \rfloor-1$, $j=\lfloor y/n_y \rfloor-1$, $u=x/n_x-\lfloor x/n_x \rfloor$, and $v=y/n_y-\lfloor y/n_y \rfloor$, and the spline basis functions $B_0(u),\dots,B_3(u)$.
\begin{figure}[t]
    \centering
    \includegraphics[width=1.0\textwidth]{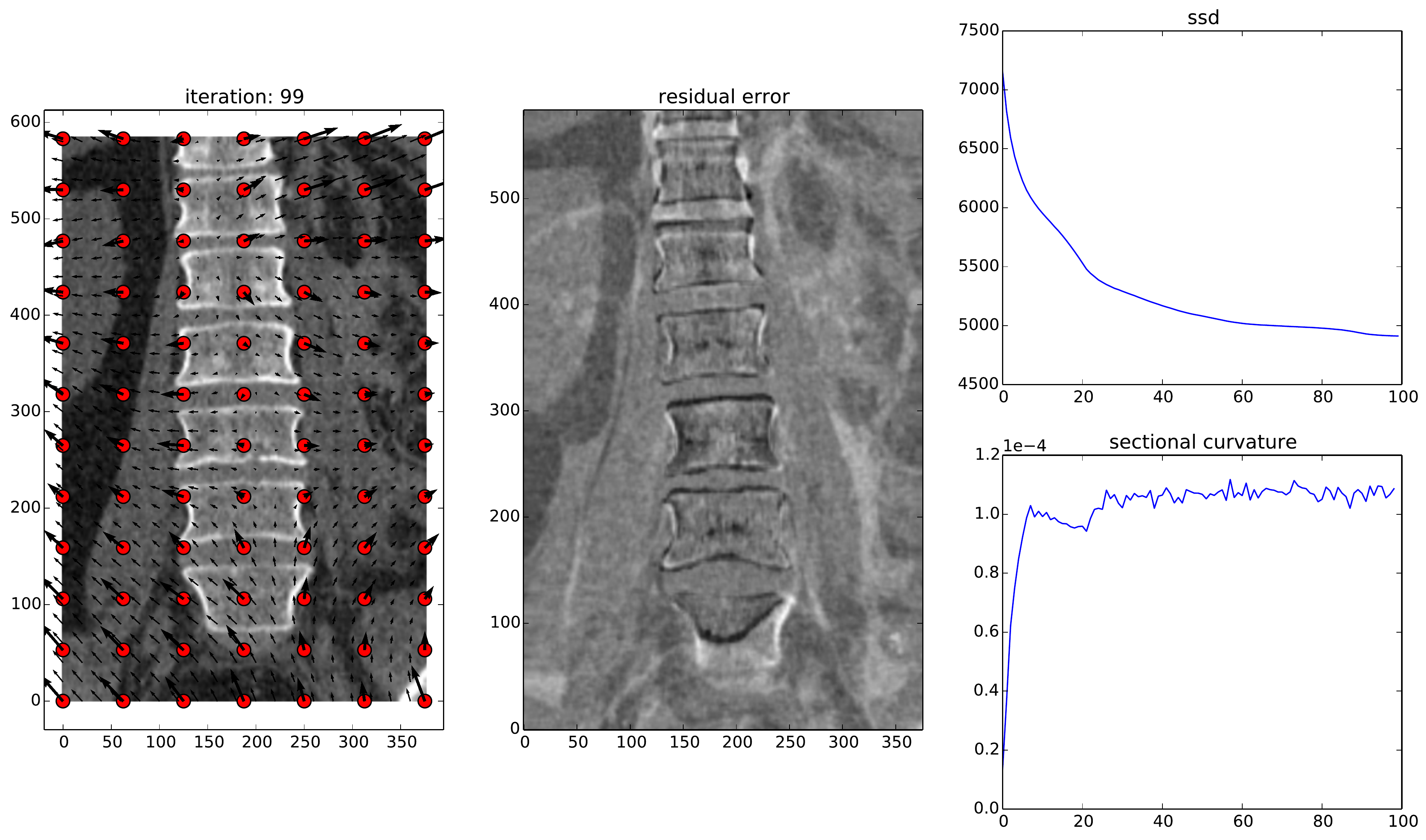}
    \caption{Left: Red control points overlaid on moving image and B-spline weight vectors. The small vector between control point are interpolated using B-splines. Middle: Difference image of deformed moving and fixed after 100 iterations. Bright pixels represent small and dark larger differences. Right: Sum of squared difference similarity metric $\sum_i (F(x_i,y_i)-M(x_i',y_i'))^2$, and mean sample sectional curvature for $d=168$ uniformly sampled orthonormal 2-frames in $\RR^d$ (see Remark \ref{rem:stiefel}) at each iteration.}
    \label{fig:ControlPoints}
\end{figure}
Figure \ref{fig:ControlPoints} shows the $12\times7$ control points that we choose for our example. This choice defines a certain amount of regularity of the deformation: more control points allow for more local flexibility.
The parameters of interest are the weights $q$ of the spline basis function at control points. In our case, we have $12 \times 7$ control points in two dimensions, which gives a total of $168$ parameters. 
In a Bayesian approach we estimate these parameters from data, which are the fixed and moving patient images, by first defining a prior probability 
\[
q \sim \mathcal{N}(0,(\lambda \Lambda)^{-1})
\]
and a likelihood 
\[
(M \mid q) \sim \frac{1}{Z} \exp\left( - \frac{\phi}{2} \sum_{i=1}^N \left( M(x_i'(q),y_i'(q)) - F(x_i,y_i) \right)^2 \right),
\]
computed over all $N$ voxels in images $F$ and $M$ at a predefined coordinate grid $(x_i,y_i)$. 
The deformed coordinates $(x_i'(q),y_i'(q))$, deformed through $\varphi^{(q)}$, usually fall between grid points $(x_i,y_i)$ and need to be linearly interpolated. 

We do not have a physical model that describes how to deform one spine to another. Such a model would only make sense when registering two images of the same patient, for instance taken before and after a surgical procedure. In that case, we could relate voxel intensities to tissue material following a mechanical law and perform mechanical simulations. The corresponding material laws from mechanics would then allow us to define a prior on the possible class of deformations. This is not possible in the absence of such a mechanical model when registering spine images from two different patients. Nevertheless, we can define a prior that is inspired by mechanics, such as the membrane energy, 
$E_m = \lambda \sum_{i\in\Omega} \sum_{j=1}^2 \sum_{k=1}^2 \left[ \partial \varphi_j / \partial x_k \right]_i$,
which measures the size of the first derivative of the deformation. To minimize $E_m$ we look for deformations with small local changes. 
The block precision matrix $\Lambda$ is given by the element-wise partial derivatives of $B$-spline basis functions, for details see \S3.5 in \cite{AJS07}.

The posterior is not analytically tractable and we need to use a Markov chain to sample from it. Since it is high dimensional, HMC is a good candidate. Simpson and coauthors recently sampled from this posterior distribution for brain images using Variation Bayes \cite{Simpson12}. Other recent related work in Bayesian approaches to image registration are \cite{Risholm2010} using Metropolis-Hastings and \cite{VanLeemput09,Zhang13} using HMC. For our example, we sample directly from the posterior distribution
\[
\pi(q) = \frac{1}{Z}\,\pi_1 (M\mid q)\,\pi_0(q)
\]
using HMC and in addition to provide concentration inequalities using our empirical curvature results. The integral of interest is
\[
I = \int_{\RR^{168}}q\,\pi(q) \, dq.
\]
We call  $J$ the Jacobi matrix that contains information about the image gradient and the spline coefficients and is of size (number of voxels) $\times$ (dimension of $q$). For details on how to construct this matrix see \S3.1 in \cite{AJS07}. Then the gradient of the potential energy $V$ is given by 
\[
\grad V = \phi J\tp r + \lambda\Lambda q.
\]
To avoid numerical problems we approximate the Hessian by Taylor expansion around the current $q$ and only keep the first order term
\[
\Hess V = \phi J\tp J + \lambda\Lambda.
\]
In contrast to the multivariate $t$ distribution, we not only need to empirically find the sectional curvature, but also approximate the Hessian of the potential. The error induced by this approximation is not considered here, but can be kept under control as long as the residual error $\|r\|$ is small relative to $\|J\tp J\|$; see \cite{GM78} and \cite{AJS07a} for details.

Figure \ref{fig:ControlPoints} shows sectional curvature numerically computed at different iterations steps $k$:
\[ 
q^{(k)} = q^{(k-1)} - (\Hess V)^{-1} \grad V.
\]
This corresponds to a Gauss-Newton minimization of the potential function $V$; see \cite{GM78} for details. If we assume that the local minimum of the potential function $V$ is an interesting mode of the posterior probability distribution $\pi$, then the sectional curvature close to that minimum will tell us how a HMC Markov chain will perform within that mode. From Figure \ref{fig:ControlPoints}, we can see that the sectional curvature is fluctuating around $10^{-4}$ after only few iterations. This is roughly the same curvature obtained in the multivariate Gaussian example, and thus the concentration inequalities carry over. 

A full analysis of HMC for an image registration application to compare the shape of spines of back pain and abdominal pain patients is in preparation; see \cite{SRH14} for details.

\section{Conclusions and Open Problems} \label{sec:conclusion}

The introduction of the Jacobi metric into the world of Markov chains promises to yield new links between geometry, statistics, and probability. 
It provides us with an intuitive connection between these fields by distorting the underlying space. 
We only scratch the surface here, and naturally we introduce some open problems:

\begin{itemize}
\item In this article, we have not focused on the numerical solving of the Hamilton equations (\ref{hamilton}), although our simulations were promising on this point. There are standard methods of solving differential equations numerically, such as the leapfrog method (see \cite{Neal11}); how might we modify Joulin and Ollivier's concentration inequality to include the parameters for the leapfrog algorithm or other algorithms?
\item Girolami and Calderhead \cite{GC11} introduced an elegant way to adapt the proposal distribution from which the momentum vector is drawn based on the underlying geometry of the distribution. Our framework can be applied also in this setting by using a non-standard reference metric. The difficulty here is to write down the expression for the sectional curvature.
\item Our approach can be extended to the setting of an infinite-dimensional Hilbert space. Since our results improve in high dimensions, we expect everything  to carry over to the infinite-dimensional setting, but we have not investigated this. Recently, there has been some work, for instance in \cite{Beskos11}, on HMC on a Hilbert space, suggesting that this topic is worthy of further study.
\item It remains to be investigated whether we can still obtain error bounds in the case of some amount of negative curvature, perhaps in a small but essential region. Alternatively, is it possible to modify the algorithm so as to give positive curvature in cases where we currently have some negative curvature?
\end{itemize}

We are currently working on estimating sectional curvatures in large-scale medical image registration projects. We believe that geometric properties of the Markov chains in such high dimensional problems can be used to assess the convergence of the chain and provides a  complement to traditional visual assessment via trace plots or statistical tests. In the context of registration, this is important since computing one step of the chain can be computationally expensive. Therefore an estimate on the number of steps becomes essential.

\bibliographystyle{alpha}
\bibliography{hmc}

\end{document}